\documentclass[11pt, reqno]{amsart}
\calclayout
\usepackage{amsmath,amsfonts,amssymb,mathrsfs,amsthm,epsfig,color,tikz,bbm,commath,systeme,mathtools,enumitem,tikz-cd,wasysym}
\usepackage{graphicx}
\usepackage{subcaption}
\usepackage[english]{babel}
\usepackage[T1]{fontenc}
\usepackage{caption}
\usepackage{verbatim} 
\usepackage{enumerate}
\usepackage{float}
\usepackage{pgfplots}
\usetikzlibrary{intersections,decorations.markings}
\usepackage{xcolor}

\usetikzlibrary{arrows.meta}
\usetikzlibrary{decorations.markings}
\usepackage{tikz-3dplot}

\usepackage[hyphens]{url}  
\usepackage[hidelinks]{hyperref}

\usepackage[
  backend=biber,
  style=alphabetic,
  doi=false,
  url=false,
  isbn=false,       
  eprint=false,     
  maxnames=99
]{biblatex}
\usepackage{etoolbox}

\newcounter{mathcounter}[section]

\makeatletter

\theoremstyle{plain}
\newtheorem{thm}{Theorem}[section]

\@addtoreset{mathcounter}{section}
\@addtoreset{thm}{mathcounter}
\let\c@thm\c@mathcounter

\newtheorem{lemma}{Lemma}

\@addtoreset{lemma}{mathcounter}
\let\c@lemma\c@mathcounter

\newtheorem{prop}{Proposition}

\@addtoreset{prop}{mathcounter}
\let\c@prop\c@mathcounter

\newtheorem{cor}{Corollary}

\@addtoreset{cor}{mathcounter}
\let\c@cor\c@mathcounter

\theoremstyle{definition}
\newtheorem{defn}{Definition}

\@addtoreset{defn}{mathcounter}
\let\c@defn\c@mathcounter

\@addtoreset{example}{mathcounter}
\let\c@example\c@mathcounter

\theoremstyle{remark}
\newtheorem{rem}{Remark}

\@addtoreset{rem}{mathcounter}
\let\c@rem\c@mathcounter

\makeatother
\makeatletter

\@addtoreset{equation}{section}
\@addtoreset{equation}{mathcounter}
\let\c@equation\c@mathcounter
\makeatother

\newtheorem*{thm*}{Theorem}
\newtheorem*{lemma*}{Lemma}
\newtheorem*{prop*}{Proposition}
\newtheorem*{cor*}{Corollary}
\newtheorem*{defn*}{Definition}
\newtheorem*{example*}{Example}
\newtheorem*{rem*}{Remark}

\newcommand{\C}{\mathbb{C}}
\newcommand{\Ha}{\mathbb{H}}

\newcommand{\N}{\mathbb{N}}
\newcommand{\Q}{\mathbb{Q}}
\newcommand{\R}{\mathbb{R}}
\newcommand{\Z}{\mathbb{Z}}

\newcommand{\E}{\mathbb{E}}
\newcommand{\V}{\mathbb{V}}
\newcommand{\Prob}{\mathbb{P}}

\newcommand{\M}{\mathcal{M}}

\newcommand{\Imm}{\mathrm{Im}}
\newcommand{\Ree}{\mathrm{Re}}

\newcommand{\GL}{\mathrm{GL}}
\newcommand{\SL}{\mathrm{SL}}

\newcommand{\PSL}{\mathrm{PSL}}
\newcommand{\minp}{\mathrm{mp}}
\newcommand{\Ta}{\mathrm{T}^1}

\newcommand{\al}{\alpha}

\newcommand{\ga}{\gamma}
\newcommand{\G}{\Gamma}

\newcommand{\eps}{\varepsilon}

\newcommand{\floor}[1]{\left\lfloor #1 \right\rfloor}

\newcommand{\bigO}{\mathcal{O}}

\title[A Central Limit Theorem for the Winding Number of Low-Lying Closed Geodesics]{A Central Limit Theorem for the Winding Number of Low-Lying Closed Geodesics}
\author{Elias Dubno}
\address{Institut f\"ur Mathematik, Universit\"at Z\"urich, Winterthurerstrasse 190, CH-8057 Z\"urich, Switzerland}
\email{elias.dubno@math.uzh.ch}

\begin{document}
\maketitle

\begin{abstract}
We show that the winding of low-lying closed geodesics on the modular surface has a Gaussian limiting distribution when normalized by any standard notion of length, in contrast to the Cauchy distribution arising when allowing arbitrarily deep excursions into the cusp. In addition, we prove a Berry-Esseen bound and a local limit theorem.
\end{abstract}

\setcounter{tocdepth}{1}
\tableofcontents
\setcounter{tocdepth}{2}

\section{Introduction}

By counting the number of oriented windings of a closed geodesic $C$ on the modular surface $\M = \PSL_2(\Z)\backslash \Ha$ about the cusp, we obtain its \textit{winding number} $\Psi(C)$. \\

In \cite{Sarnak2010} (see also \cite{Mozzochi2012}), Sarnak obtained precise counting results for the number of primitive (i.e. once around) closed geodesics on the modular surface of a given winding number. (In fact, he framed his results for linking numbers of modular knots with the trefoil knot, which is an equivalent characterization of the winding about the cusp; we will review this object in more detail in the second part of this introduction.) \\

Moreover, Sarnak proved that when ordering closed geodesics by their geometric length, the ratio of winding to length has a limiting Cauchy distribution; that is, for any $-\infty\leq a < b \leq +\infty$, we have\\
\begin{align*}
\lim_{T\to\infty}\frac{\abs{\{\text{primitive closed geodesics } C\text{: } \ell_g(C)\leq T,~ a\leq \frac{\Psi(C)}{\ell_g(C)}\leq b \}}}{\abs{\{\text{primitive closed geodesics } C\text{: } \ell_g(C)\leq T\}}}  & =  \frac{1}{\pi} \int_a^b \frac{\gamma}{x^2 + \gamma^2} \, dx \\
& = \frac{\arctan\left(\frac{b}{\gamma}\right) - \arctan\left(\frac{a}{\gamma}\right)}{\pi},
\end{align*}
with $\gamma = \frac{3}{\pi}$. Here, $\ell_g(C)$ denotes the geometric (or hyperbolic) length of $C$. Burrin--von Essen \cite{BurrinVonEssen} have generalized Sarnak's results for a wide class of Fuchsian groups. \\

The goal of this paper is to show that once arbitrarily deep cusp excursions are excluded, the heavy tails of the Cauchy law disappear and are replaced by a Gaussian law. 

\begin{rem}
    The Cauchy distribution is the prototype example of a ``pathological'' distribution in the sense that its fat tails imply the divergence of all moments. It also shows up in the related context of Dedekind sums \cite{Vardi1993}, and the common probabilistic mechanism behind the appearance of the Cauchy distribution in both instances is nicely explained in \cite{kowalski-modcauchy}.
\end{rem}

The fat tails of the Cauchy distribution arising here can be explained geometrically by the following heuristic: whenever a closed geodesic goes high in the cusp, its winding number will change drastically compared to its length due to the nature of the hyperbolic metric \cite[Section 1]{BurrinVonEssen}. So in essence, the Cauchy distribution should come from the presence of the cusp.\\

There are various results supporting this heuristic: Calegari--Fujiwara \cite{CalegariFujiwara2009} (see also \cite[Section 6]{Calegari2009stable}) obtained a limiting Gaussian distribution when ordering and normalizing closed geodesics by word length $\ell_w$ using the fact that $\PSL_2(\Z) \cong \Z_2 * \Z_3$. In a slightly different setting, Guivarc'h--Le Jan \cite{Guivarch1993} considered the winding in homology of the geodesic flow. They showed that the limiting distribution is Gaussian when the corresponding harmonic 1-form comes from a cusp form, and Cauchy otherwise. \\

Our plan here is to ``remove the cusp'' by focusing on \textit{low-lying} geodesics. We can do so via a simple Diophantine condition using the well-established connection between the cutting sequence of the geodesic flow and continued fractions \cite{Artin1924, Series1985}. More concretely, unfolding a closed geodesic $C$ from the modular surface produces a geodesic on $\Ha$ connecting two quadratic irrationals which have eventually periodic continued fraction expansions. Although such a lift of $C$ is not unique, the minimal even length periodic part $[\overline{a_1,a_2,\dots,a_n}]$ of one of its endpoints is, and it characterizes $C$. We say that the \textit{period length} of $C$ is $\ell_p(C)=n$. \\

Moreover, this periodic part directly encodes geometric data of $C$. The size of the partial quotients ``measures'' how far the geodesic travels into the cusp, so closed geodesics with \emph{uniformly bounded partial quotients} correspond to low-lying geodesics. This perspective also features in Bourgain--Kontorovich \cite{BourgainKontorovich2017}, where they settle a question of Einsiedler--Lindenstrauss--Michel--Venkatesh \cite{Einsiedler2009} by proving the existence of infinitely many low-lying fundamental geodesics.   

\begin{defn}
    We say that a closed geodesic $C$ is $A$\textit{-low-lying} if all partial quotients in the periodic part associated to $C$ satisfy $a_i\leq A$, and we define $$\Pi_A(N)= \{\textnormal{primitive } A\textnormal{-low-lying closed geodesics } C\text{: } \ell_p(C)\leq N\}$$ 
and $$\pi_A(N) = \abs{\Pi_A(N)}.$$
\end{defn}

When considering $A$-low-lying geodesics for some fixed $A$, it is most natural to order them by period length, which differs from the usual ordering by geometric or word length. Our first main result is the following.

\begin{thm}[Main Central Limit Theorem]\label{thm:mainp}
    Let $A>1$. The winding-to-period length ratio of $A$-low-lying closed geodesics has a limiting Gaussian distribution when ordered by period length. More precisely, for $\sigma_p^2=\sigma^2_p(A) = \frac{A^2-1}{12}$ and any $-\infty\leq a < b \leq +\infty$, we have
\begin{align*}
\lim_{N\to\infty}\frac{\abs{\{C\in\Pi_A(N): ~ a \leq \frac{\Psi(C)}{\sqrt{\ell_p(C)}} \leq b\}}}{\pi_A(N)}  & =  \frac{1}{\sqrt{2\pi\sigma_p^2}} \int_{a}^{b}  e^{-x^2/2\sigma_p^2} \, dx.
\end{align*}
\end{thm}

We also establish two natural strengthenings of this central limit theorem. First of all, we can prove a Berry--Esseen bound, providing an explicit rate of convergence.

\begin{thm}[Berry--Esseen]\label{thm:BE}
    Let $F_N$ and $\Phi=\Phi_{\sigma_p}$ be the cumulative distribution function of $\frac{\Psi(\cdot)}{\sqrt{\ell_p(\cdot)}}$ on $\Pi_A(N)$ and $\mathcal{N}(0,\sigma_p^2)$, respectively. Then

    $$\sup_{x\in\R}\abs{F_N(x)-\Phi(x)} = \bigO\left(\frac{1}{\sqrt{N}}\right).$$
\end{thm}

There is also a local limit theorem describing the fine-scale distribution of the winding number. 

\begin{thm}[Local Limit Theorem]\label{thm:llt}
    We have 
     $$\frac{\#\{C\in\Pi_A(N): \Psi(C) = k\}}{\pi_A(N)} = \frac{1}{\sqrt{2\pi\sigma_p^2 N}}\exp\left(-\frac{k^2}{2\sigma_p^2 N}\right) + o\left(\frac{1}{\sqrt{N}}\right),$$
    uniformly in $k=\bigO\left(\sqrt{N}\right)$.
\end{thm}

Although period length is a natural ordering parameter for low-lying geodesics via continued fractions, it is non-standard from a geometric or group-theoretic perspective. A priori, ordering by different notions of length could lead to different limiting behavior. The purpose of the following comparison theorem is to show that this does not occur: for low-lying geodesics, all standard notions of length are asymptotically equivalent, and the resulting Gaussian limit is therefore independent of the chosen standard length. \\

More precisely, other natural length functions one can associate to a closed geodesic $C = [\overline{a_1,\dots,a_n}] \in \Pi_A(N)$ include the \emph{geometric} (or \emph{hyperbolic}) length $\ell_g(C)$ of the geodesic, the \emph{word} length $\ell_w(C)$ of the associated hyperbolic conjugacy class in $\PSL_2(\Z)$, and the \emph{maximal} length $\ell_{\max}(C) \equiv N$. \\

\begin{defn}
We call a length function $\ell_*$ \emph{standard} if
\begin{equation*}
\ell_* \in \{ \ell_g, \ell_w, \ell_{\max} \}.
\end{equation*}
\end{defn}

On $\Pi_A(N)$, all standard length functions are asymptotically equivalent to the period length.\\

\begin{thm}[Comparison Theorem]\label{thm:comparison}
Let $\ell_*$ be a standard length. Then there exists a constant $c_* > 0$ such that
\begin{equation*}\label{eq:comparison}
\frac{\ell_*}{\ell_p} \overset{\mathbb{P}}{\longrightarrow} c_*
\qquad \text{as } N \to \infty \text{ on } \Pi_A(N).
\end{equation*}
\end{thm}

The analysis of $\frac{\ell_g}{\ell_p}$ is the most substantial and might be of independent interest; the key step is an approximation scheme using truncated continued fractions, leading to the convergence result via the Ergodic Theorem for Markov chains.

\begin{rem}Statistical questions about the ratios of different notions of lengths have been studied by various authors before. For example, see \cite{SpaldingVeselov2017} and more recently \cite{BIH2025} for an analysis of such ratios for Markov geodesics. In another setting, Pollicott--Sharp \cite{PollicottSharp} showed that for closed geodesics on compact smooth surfaces, the ratio $\frac{\ell_g}{\ell_w}$ converges \emph{on average} to a constant when ordered by geometric length. More recently, Cantrell--Pollicott \cite{CantrellPollicott2022} even established an error term for this convergence. Although the settings are not directly comparable, we emphasize that our Comparison Theorem \ref{thm:comparison} is stronger than such a convergence on average. 
\end{rem}

As an immediate corollary of Theorem \ref{thm:comparison}, we obtain central limit theorems when normalizing by any standard length. The variances depend on the choice of normalization. 

\begin{thm}[CLT with different lengths]\label{thm:otherlengths}
Let $A>1$ and let $\ell_*$ be a standard length. There exists
$\sigma_*^2 = \sigma^2(A,\ell_*) > 0$ such that for any
$-\infty \le a < b \le +\infty$,
\begin{equation*}\label{eq:clt-other-length}
\lim_{N\to\infty}
\frac{\abs{\{ C \in \Pi_A(N) : a \leq \frac{\Psi(C)}{\sqrt{\ell_*(C)}} \leq b \}}}{\pi_A(N)}
=
\frac{1}{\sqrt{2\pi\sigma_*^2}}
\int_a^b e^{-x^2/2\sigma_*^2}\,dx .
\end{equation*}
\end{thm}

\begin{figure}[t]
    \centering
    \includegraphics[width=\textwidth]{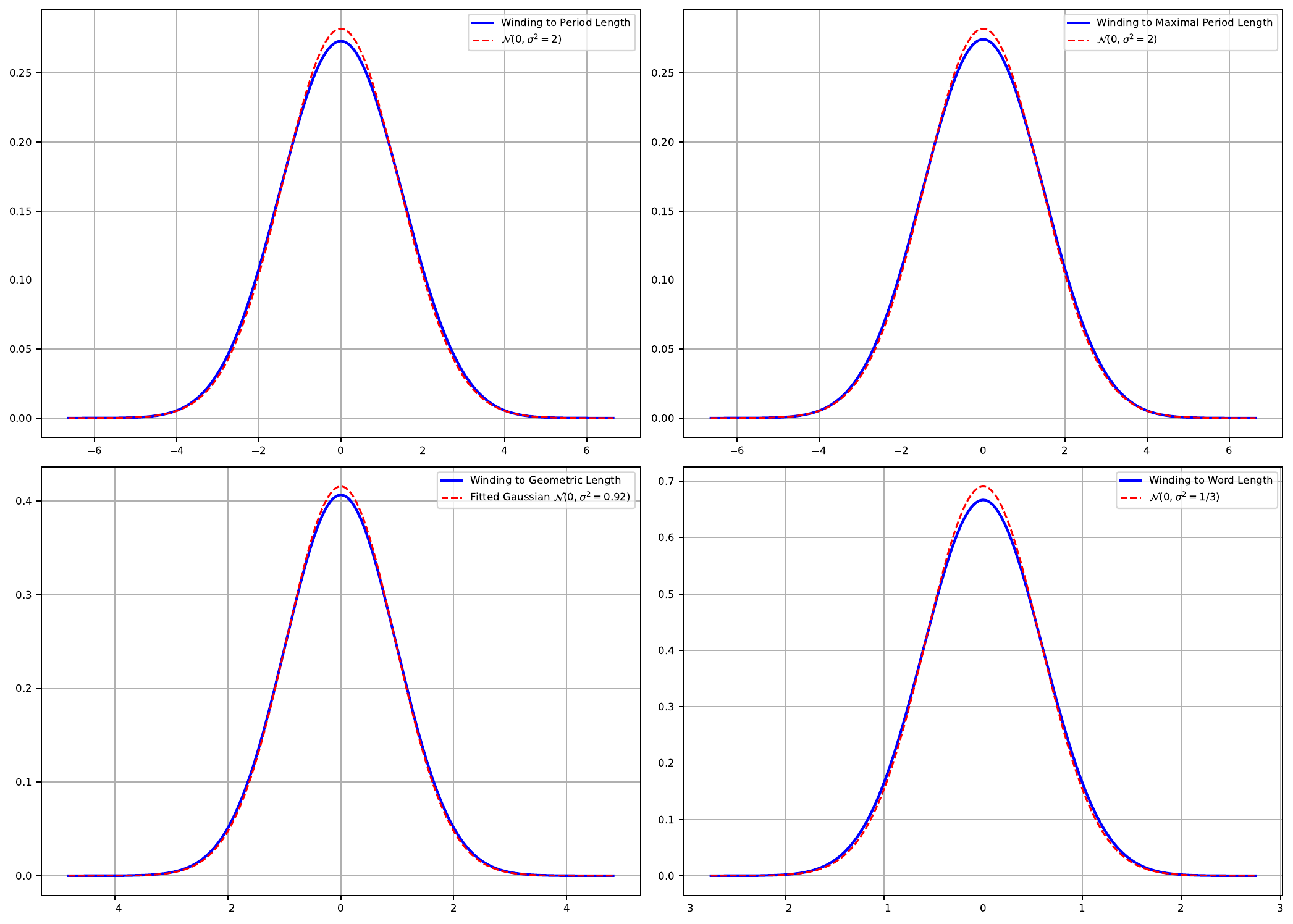}
    \caption{Illustration of Main Results for $A=5$. Empirical distributions are shown in blue and the corresponding limiting Gaussians in dashed red.}
    \label{fig:winding-distributions}
\end{figure}

\begin{rem}\label{rem:explicitvariances}
For the normalizations considered above, the variances can be made explicit in several cases. 
In particular, we have
$$
\sigma_{\max}^2 = \sigma_p^2 = \frac{A^2-1}{12}
\qquad \text{and} \qquad
\sigma_w^2 = \frac{A-1}{12}.
$$
The only variance for which we do not obtain a closed-form expression is $\sigma_g^2$; see Section \ref{sec:comparison} for proofs and further discussion. \\

We note that probabilistic results of this type can also be obtained using methods from the study of subshifts of finite type, such as the machinery of thermodynamic formalism. In particular, more general Berry--Esseen bounds and local limit theorems are available in the work of Coelho–Parry \cite[Theorem 1]{Coelho1990} and Pollicott–Sharp \cite[Theorem 1]{PollicottSharp1994}, respectively. The approach we take here is more natural for the present setting and exploits the explicit nature of the low-lying geodesics under consideration, allowing a direct and self-contained treatment.
\end{rem}
    
Theorem $\ref{thm:otherlengths}$ (applied to $\ell_g$) is in stark contrast to Sarnak's result \cite{Sarnak2010}, illustrating how the restriction to low-lying closed geodesics affects the limiting distribution. Moreover, we point out that the correct normalization for us involves taking a square root, in contrast to the normalization used by Sarnak. A similar dichotomy appears in \cite[Theorem 2.1]{Guivarch1993}, where square root normalization leads to a Gaussian distribution, while linear normalization yields a Cauchy distribution. \\

In the four plots of Figure \ref{fig:winding-distributions}, we empirically compare Theorems \ref{thm:mainp} and \ref{thm:otherlengths} with the corresponding theoretical distributions. The empirical distributions are shown in solid blue, while the theoretical Gaussians are drawn in dashed red lines.\\

We have computed the data for $A = 5$, for which we have
$$\sigma_p^2 = \frac{5^2-1}{12}=2 ,~~\sigma^2_g  \approx 0.92\dots, ~ \sigma_w^2 = \frac{5-1}{12}=\frac{1}{3};$$
see Remark \ref{rem:explicitvariances}.\\

We have considered period lengths up to $N = 12$, producing a total of $\pi_5(12) = \text{42'743'545}$ geodesics.\\

The resulting number of geodesics aligns closely with the theoretical prediction of Proposition \ref{prop:2}, which gives an expected total of
$$
\frac{2\cdot25}{24} \cdot \frac{5^{12}}{12} \approx \text{42'385'525},
$$
with a small relative error of approximately $0.84\%$. \\

We now give a more concrete description of the winding number using Figure \ref{fig2}. There, the dashed blue segment is part of the geodesic arc connecting the two Galois conjugates $$w,w' = \frac{8\pm\sqrt{84}}{5}$$ with purely periodic continued fraction $w=[\overline{3,2,3,4}]$. Traveling from left to right along the dashed blue segment and continuing along the blue parts of the closed geodesic, we count $a_1=3$ windings around the cusp before hitting the unit circle. Afterwards, we flow along the red, green and white arcs, alternating directions and starting with the dashed line each time after hitting the unit circle. We indeed record $a_2=2$, $a_3=3$ and $a_4=4$, which illustrates the fact that the winding number is nothing else than the alternating sum of the partial quotients. In the above example, the winding number is $\Psi([\overline{3,2,3,4}]) = 3-2+3-4 = 0$, which is also apparent from the symmetry along the $y$-axis. Moreover, the figure illustrates how the size of the partial quotients correlates with the depth of the excursion to the cusp. \\

\begin{figure}[h]
\centering

\begin{tikzpicture}[scale=4.2]

	\clip(-0.5,0.86) rectangle (0.5,3);
    \clip (-1,0) arc[start angle=180, end angle=0, radius=1] -- (1,0) -- (1,6) -- (-1,6) -- cycle;

\begin{scope}
  \clip (-0.5,0.866) -- (-0.5,3) -- (0.5,3) -- (0.5,0.866) -- cycle;
  \clip (-1,0) arc[start angle=180, end angle=0, radius=1] -- (1,0) -- (1,3) -- (-1,3) -- cycle;
  \fill[gray!40] (-1,0) rectangle (1,3);
\end{scope}

\draw[thick] (-0.5,0) -- (-0.5,3);
\draw[thick] (0.5,0) -- (0.5,3);
\draw[thick] (-1,0) arc[start angle=180, end angle=0, radius=1];



  \def\centerX{8/5}
  \def\radius{sqrt(84)/5}

  \draw[blue, ultra thick, dashed] 
    ({\centerX - \radius}, 0) 
    arc[start angle=180, end angle=0, radius=\radius];

    \def\centerX{8/5-1}
  \draw[blue, ultra thick] 
    ({\centerX - \radius}, 0) 
    arc[start angle=180, end angle=0, radius=\radius];

    \def\centerX{8/5-2}
  \draw[blue, ultra thick] 
    ({\centerX - \radius}, 0) 
    arc[start angle=180, end angle=0, radius=\radius];

    \def\centerX{8/5-3}
  \draw[blue, ultra thick] 
    ({\centerX - \radius}, 0) 
    arc[start angle=180, end angle=0, radius=\radius];

    \def\centerX{-1}
  \def\radius{sqrt(84)/7}

    \draw[red, ultra thick, dashed] 
    ({\centerX - \radius}, 0) 
    arc[start angle=180, end angle=0, radius=\radius];

    \def\centerX{-1+1}
    \draw[red, ultra thick] 
    ({\centerX - \radius}, 0) 
    arc[start angle=180, end angle=0, radius=\radius];

    \def\centerX{-1+2}
    \draw[red, ultra thick] 
    ({\centerX - \radius}, 0) 
    arc[start angle=180, end angle=0, radius=\radius];

    \def\centerX{7/5}
  \def\radius{sqrt(84)/5}
  \draw[green!90!black, ultra thick, dashed] 
    ({\centerX - \radius}, 0) 
    arc[start angle=180, end angle=0, radius=\radius];

    \def\centerX{7/5-1}
  \draw[green!90!black, ultra thick] 
    ({\centerX - \radius}, 0) 
    arc[start angle=180, end angle=0, radius=\radius];
 
    \def\centerX{7/5-2}
  \draw[green!90!black, ultra thick] 
    ({\centerX - \radius}, 0) 
    arc[start angle=180, end angle=0, radius=\radius];

    \def\centerX{7/5-3}
  \draw[green!90!black, ultra thick] 
    ({\centerX - \radius}, 0) 
    arc[start angle=180, end angle=0, radius=\radius];

    \def\centerX{-2}
  \def\radius{sqrt(21)/2}
  \draw[white, ultra thick, dashed] 
    ({\centerX - \radius}, 0) 
    arc[start angle=180, end angle=0, radius=\radius];

    \def\centerX{-2+1}
  \draw[white, ultra thick] 
    ({\centerX - \radius}, 0) 
    arc[start angle=180, end angle=0, radius=\radius];

    \def\centerX{-2+2}
  \draw[white, ultra thick] 
    ({\centerX - \radius}, 0) 
    arc[start angle=180, end angle=0, radius=\radius];

    \def\centerX{-2+3}
  \draw[white, ultra thick] 
    ({\centerX - \radius}, 0) 
    arc[start angle=180, end angle=0, radius=\radius];

    \def\centerX{-2+4}
  \draw[white, ultra thick] 
    ({\centerX - \radius}, 0) 
    arc[start angle=180, end angle=0, radius=\radius];
\end{tikzpicture}\\

\caption{The closed geodesic $[\overline{3,2,3,4}]$.}
\label{fig2}
\end{figure}

In general, for a closed geodesic $C$ with associated minimal even-length periodic part $[\overline{a_1,\dots,a_n}]$, its winding number is given by the alternating sum $$\Psi(C) = \Psi([\overline{a_1,\dots,a_n}]) = a_1-a_2+-\dots+a_{n-1}-a_n.$$ 

The winding number has various other interpretations; one of them is the \textit{Rademacher symbol}, which motivates our notation $\Psi$ and arises from the transformation law of the Dedekind eta function \cite{Rademacher1972}. Alternatively, the winding number can be viewed as  real period $$\Psi(C) = \int_C E_2(z)\, dz,$$ of a non-holomorphic Eisenstein series of weight 2; see \cite{Burrin2023}. Lastly, as mentioned at the beginning, the winding number can be interpreted as a linking number. In fact, it has been known at least since Milnor \cite{Milnor1972} that the unit tangent bundle of the modular surface $\Ta\M$ is homeomorphic to $S^3 \setminus \tau$, where $\tau$ is a trefoil knot. Ghys \cite{Ghys2007} proved that a closed geodesic $C$ on $\M$ gives rise to a knot $k_C$ in the complement of $\tau$, and $\Psi(C) = \mathrm{link}(k_C,\tau)$ is the linking number between $k_C$ and the trefoil. Various generalizations in this direction were studied in \cite{DIT2017}. \\

An interesting question, raised independently by Emmanuel Kowalski and Toshiki Matsusaka, is what happens when the truncation parameter $A$ is allowed to grow with $N$, i.e. $A = A(N)$. With the exception of the geometric length normalization, our arguments extend with only minor modifications, including the counting result and the distributional statements normalized by period length or word length. The geometric length case is genuinely different. The heuristic discussed at the beginning of the introduction suggests a transition from the Gaussian distribution obtained here to the Cauchy distribution observed by Sarnak as $A(N) \to \infty$. Our method breaks down when $A$ is unbounded, since the Markov chain argument used in Comparison Theorem \ref{thm:comparison} does not apply in this setting, and we plan to return to this question in future work. \\ 

In a different direction, it would also be interesting to make Theorem \ref{thm:comparison} quantitative by obtaining effective rates of convergence.

\subsection*{Structure of the paper}
In Section \ref{sec:background}, we provide background on closed geodesics and continued fractions. Next, in Section \ref{sec:lowlying} we analyze the structure and asymptotic size of $\Pi_A(N)$, before proving the main central limit theorem in Section \ref{sec:mainproof}. This is done by exploiting the connection to continued fractions and reducing the problem to a classical CLT setting. In Sections \ref{sec:fixed} and \ref{sec:varying} we then establish two strengthenings of this result, namely a Berry--Esseen bound and a local limit theorem, by first proving both statements for fixed period length and then combining them for $\Pi_A(N)$. Finally, in Section \ref{sec:comparison} we prove the comparison theorem and use it to deduce the corresponding central limit theorems for the remaining notions of length.

\subsection*{Notation} Let $f$ and $g$ be two real-valued functions, and let $X,Y$ be random variables.
\begin{itemize}[label=\textendash{}]
    \item We write $f(x) = \bigO(g(x))$ or $f(x) \ll g(x)$ as $x\to a$ if $\exists C>0$ such that $\abs{f(x)} \leq C\abs{g(x)}$ for all $x$ sufficiently close to $a$. \
    \item We write $f(x) \asymp g(x)$ if $f(x) \ll g(x) \ll f(x)$.
    \item We write $f(x)\sim g(x)$ if $\frac{f(x)}{g(x)}\to 1$ as $x\to\infty$. If, in addition, $\sum_{n=1}^N g(n)\to\infty$, then $\sum_{n=1}^N f(n)\sim \sum_{n=1}^N g(n)$ as $N\to\infty$.
    \item We write $f(x) \approx g(x)$ if $f(x) - g(x) \longrightarrow 0$ as $x\rightarrow\infty.$
    \item For a positive integer $A$, we define $[A] \vcentcolon = \{1,...,A\}$.
    \item $G = \PSL_2(\R)$, $\G = \PSL_2(\Z)$.
    \item $\E(X)$ and $\V(X)$ denote the expectation and variance of $X$. 
    \item $\Prob(\cdot)$ denotes the probability of some event. 
    \item $X \overset{d}{\longrightarrow} Y$ and $X \overset{\Prob}{\longrightarrow} Y$ denote convergence of $X$ to $Y$ in distribution and probability, respectively.
    \item $\mathcal{N}(\mu,\sigma^2)$ denotes the normal distribution with expectation $\mu$ and variance $\sigma^2$.
\end{itemize}

\subsection*{Acknowledgments} I would like to thank my advisor Claire Burrin for introducing me to this problem and for her patient, insightful, and encouraging support. I am also grateful to Zheng Fang, Michael Schaller and Andrea Ulliana for several helpful conversations. Finally, I thank Emmanuel Kowalski and Toshiki Matsusaka for suggesting the problem of allowing $A=A(N)\to\infty$, which I plan to return to in upcoming work. This work was supported by Swiss National Science Foundation grant number 201557.

\section{Background}\label{sec:background}

As usual, let $\Ha=\{z=x+iy\in \C: y>0\}$ be the Poincaré upper half-plane, and let $\Ta\Ha = \{(z,v): z\in\Ha, ~ \norm{v}_z=\frac{\abs{v}}{\Imm(z)}=1\}$ denote the unit tangent bundle of $\Ha$. The group $G=\mathrm{PSL}_2(\R)$ acts freely and transitively on $\Ta\Ha$ via 
\begin{equation}\label{eq:1.1}
    \begin{pmatrix}
    a &b \\ c & d
\end{pmatrix}  (z,v) \mapsto \left( \frac{az+b}{cz+d}, \frac{v}{(cz+d)^2}\right).
\end{equation}
This action is invariant under the measure $d\mu = \frac{dxdyd\theta}{y^2}$, where $x=\Ree(z)$, $y=\Imm(z)$ and $\theta=\mathrm{arg}(v)$. For more details, we refer to \cite{Einsiedler2011, Kontorovich2016}.\\ 

Any element in $\Ta\Ha$ can be written as $g(i,i)$ for a unique $g\in G$. By identifying $\Ta\Ha$ and $G$ in this manner, the geodesic flow on $\Ta\Ha$ amounts to multiplication from the right by the diagonal subgroup 
$$A=\left\{\begin{pmatrix}
    e^{t/2}&0 \\ 0 & e^{-t/2}
\end{pmatrix}: t\in\R \right\}.$$

To study closed geodesics on the modular surface $\M = \G\backslash\Ha$, we observe that the action \eqref{eq:1.1} descends to $\Ta\M \cong \G\backslash G$. Closed geodesics then correspond to periodic orbits under the action of $A$. While this holds for general Fuchsian groups, we focus on $\Gamma = \PSL_2(\Z)$ throughout the paper. \\

Another way of representing a closed geodesic is to note that it corresponds to an equivalence class $\G g$ of some \textit{hyperbolic} matrix, that is, its trace satisfies 
\begin{equation}\label{eq:hyperbolic}
    \abs{\text{tr}(g)} > 2.
\end{equation}

Such $g$ acts via linear fractional transformation on the Riemann sphere and \eqref{eq:hyperbolic} ensures that the two fixed points of said action are distinct real numbers. Indeed, we have 

\begin{align}\label{eq:moebius}
    g w = w & \iff \frac{aw+b}{cw+d} = w \nonumber\\ 
    & \iff cw^2 + (d-a)w -b = 0 \\ 
    & \iff w,w' = \frac{a-d \pm \sqrt{\text{tr}(g)^2-4}}{2c},\nonumber
\end{align}

for $g=\pm \begin{pmatrix}
        a & b \\ c & d
    \end{pmatrix}$, and the two fixed points form a pair of Galois conjugate quadratic irrationals.\\
    
We may also obtain that closed geodesic on $\M$ by placing a geodesic on $\Ha$ joining these two fixed points and projecting it down to the modular surface. One can easily verify that this is well defined, that is, the resulting closed geodesic is independent of our choice of $g$ in its equivalence class. \\

Any $w\in\R$ can be written as a continued fraction 
$$w = a_1 + \cfrac{1}{a_2 + \cfrac{1}{a_3 + \cfrac{1}{a_4 + \cdots}}} = [a_1,a_2,a_3,a_4,\dots]
$$
with $a_1\in\Z$ and $a_i\in\Z_{>0}$ for $i>1$. The $a_i$'s are called \emph{partial quotients} of $w$.  Since $w$ is a quadratic irrational, this expansion is $\textit{eventually}$ periodic, and by applying a certain $\G$-action , we may assume that $w$ is \textit{reduced}; that is, the fixed points satisfy $-1 < w' < 0 < 1 < w$ and thus the continued fraction expansion of $w=[\overline{a_1,\dots,a_n}]$ is \textit{purely} periodic. (In that case, it is well-known that $-w' = [0,\overline{a_n,\dots,a_1}]$.) Moreover, by doubling the period if necessary, we may assume that $n$ is even.

\section{Low-Lying Closed Geodesics}\label{sec:lowlying}

For the remainder of the paper, we now fix an integer $A>1$. Whenever we talk about low-lying closed geodesics, we mean $A$-low-lying; that is, all the partial quotients in its periodic part satisfy $a_i\leq A$. \\

The purpose of this section is to study the set $\Pi_A(N)$ of low-lying geodesics of period length up to $N$ and to give an asymptotic count for its size $\pi_A(N)$. \\

We repeat that we can attach a finite sequence $\al = (a_1,\dots,a_n)$ of even length to any closed geodesic. So, as a first guess, one might think that we should consider all even-length finite sequences in the alphabet $[A] = \{1,\dots,A\}$. However, there are two problems: \emph{primitivity} and \emph{even cyclic shifts}. 

\subsection*{Primitivity} We have to take into account the minimal period length of the sequences. For example, the sequences $(1,2)$, $(1,2,1,2)$, and $(1,2,1,2,\dots,1,2)$ all correspond to the same periodic part and thus to the same closed geodesic, so it makes sense to consider only the sequence $(1,2)$. \\

It is easy to see --- but crucial --- that for any sequence $\al=(a_1,\dots,a_n)$, the minimal period length $\minp(\al)$ is a divisor of $n$.

\begin{defn}
A sequence $\alpha = (a_1, \dots, a_n)$ is \emph{primitive} if its minimal period length $\minp(\alpha)$ equals $n$. In the case where $\frac{n}{2}$ is odd, we also consider $\alpha$ as primitive if $\minp(\alpha) = \frac{n}{2}$.
\end{defn}

Consider the following example: Assume that $A=2$, and let's say we want to find all primitive sequences for $n=2$. Then we must consider all sequences of minimal period length $2$ and $1$, i.e. the sequences $(1,2)$ and $(2,1)$, as well as $(1,1)$ and $(2,2)$. When $n=4$, we consider only sequences of minimal period length $4$, of which there are 12.

\begin{prop}\label{prop:1}
    Let $f(n)$ denote the number of sequences $\al\in[A]^n$ with $\minp(\al) = n$. Then $$f(n) = \sum_{k\mid n}\mu(k)A^{n/k} = A^n + \bigO(A^{n/2}),$$ where $\mu(\cdot)$ denotes the Möbius function. 
\end{prop}

\begin{proof}
     Fix $n\geq 1$. Note that the number of elements in $[A]^n$ of primitive period $k$ is equal to $f(k)$ if $k$ divides $n$, and $0$ otherwise. It follows that the sum of $f(k)$, where $k$ runs over all divisors of $n$, should give the total number of $n$-tuples, so

    \begin{equation}\label{eq:sumofgks}
    \sum_{k\mid n} f(k) = A^n.
    \end{equation}

    By Möbius Inversion, it follows that $$ f(n) = \sum_{k\mid n}\mu(k)A^{n/k}.$$

    Note that \begin{align*}
        \abs{\sum_{k\mid n}\mu(k)A^{n/k}-A^n} & \leq \sum_{\substack{k\mid n \\ k<n}} A^k \leq \sum_{k=1}^{\floor{n}/2}A^k = \bigO (A^{n/2}),
    \end{align*}
    which gives the error term.
\end{proof}

It follows directly that the same asymptotic count holds for primitive sequences.

\begin{cor}\label{cor:fewnonprim}
    The number of primitive sequences in $[A]^n$ is $$A^n + \bigO(A^{n/2}).$$ Indeed, when $\frac{n}{2}$ is odd, the additional primitive sequences of minimal period $\frac{n}{2}$ contribute at most $A^{n/2}$, which is absorbed into the error term.
\end{cor}

\subsection*{Even Cyclic Shifts} The purely periodic quadratic irrational $w=[\overline{a_1,a_2\dots,a_n}]$ with $n$ even is a fixed point of the hyperbolic matrix \begin{equation}\label{eq:matrixproduct}
    \pm \ga = \begin{pmatrix}
    a_1 & 1 \\ 1 & 0
\end{pmatrix}\begin{pmatrix}
    a_2 & 1 \\ 1 & 0
\end{pmatrix}\dots\begin{pmatrix}
    a_n & 1 \\ 1 & 0
\end{pmatrix}.
\end{equation} Note that $g\in\G$ since $\det(g) = (-1)^n = 1$. From this presentation, it follows that cyclic shifts of a given periodic part yield $\GL_2(\Z)$-conjugates of the corresponding matrix and thus to $\GL_2(\Z)$-equivalent geodesics on $\Ha$. On the other hand, \textit{even} cyclic shifts correspond to $\SL_2(\Z)$-equivalent geodesics, which project down to the same closed geodesic on the modular surface; see also \cite[Section 4]{BurrinVonEssen}. \\

By identifying any two primitive sequences $\al=(a_1,\dots,a_n)$ and $\beta=(b_1,\dots,b_n)$ if one can be obtained by applying a cyclic shift of even length to the other, we get an equivalence relation $\sim$. We define $$P_n \vcentcolon= \big\{\al\in[A]^n\mid ~ \al \text{ primitive}\big\}/\sim,$$ and for any $\al\in P_n$, we define its \textit{period length} to be $$\ell_p(\alpha) = n.$$

In this way, $P_n$ is in one-to-one correspondence with the set of primitive closed geodesics of period length $n$.

\begin{rem}
    In combinatorics, finite sequences modulo cyclic rotations are known as \emph{necklaces} \cite{Berstel2008}, and thus we will sometimes refer to elements of $P_n$ as necklaces. Using similar combinatorial approaches, Basmajian--Suzzi Valli \cite{BasmajianSuzziValli2022} and Basmajian--Liu \cite{BasmajianLiu2023} gave asymptotic counts for the number of reciprocal or low-lying closed geodesics on the modular surface, using the fact that $\G \cong \Z_2 * \Z_3$ and ordering by word length with respect to two generators of the individual factors. 
\end{rem}

For each necklace $\al\in P_n$, there are exactly $\frac{n}{2}$ corresponding (via even cyclic shifts) primitive sequences of length $n$, and thus

\begin{equation}\label{eq:f}
\abs{P_n} = 
    \frac{f(n)}{n/2} = \frac{2A^n}{n}+ \mathcal{O}\left(\frac{A^{n/2}}{n}\right)
\end{equation}

We can identify $\Pi_A(N)$ with $$\{\text{primitive necklaces of period length} \leq N \} = \bigcup_{n\leq N \text{ even}} P_n,$$ and because the union of $P_n$ is disjoint, we can sum over the values of \eqref{eq:f} to obtain an expression for $\pi_A(N)$.

\begin{prop}\label{prop:2}
    Let $N\to \infty$ along even integers. Then $$\pi_A(N) \sim c_A\frac{A^{N}}{N},$$
    where $c_A = \frac{2A^2}{A^2-1}$.
\end{prop}

The proof will follow directly once we have established the following counting results.

\begin{lemma}\label{lem:sum}
    Let $A>1$. As $N\to\infty$, we have $$\sum_{n=1}^{N} \frac{A^n}{n} = \frac{A}{A-1}\frac{A^N}{N} + \bigO\left(\frac{A^{N}} {N^2}\right).$$
\end{lemma}

\begin{proof}
The main term follows from Stolz--Cesàro. To get the error term, one can use summation by parts.
\end{proof}
    
Moreover, we require an estimate for the same sum along even integers. 

\begin{lemma}\label{lem:sumeven}
    Let $A>1$. As $N\to\infty$ along even integers, we have $$\sum_{\substack{n=2\\ \textnormal{even}}}^{N} \frac{A^n}{n} = \frac{A^2}{A^2-1} \frac{A^N}{N} + \bigO\left(\frac{A^{N}} {N^2}\right).$$
\end{lemma}

\begin{proof}
Follows directly from Lemma \ref{lem:sum} after replacing $A$ by $A^2$ and $N$ by $N/2$.
\end{proof}

By summing over \eqref{eq:f} in combination with Lemma \ref{lem:sumeven}, it is straight-forward to deduce the asymptotic count from Proposition \ref{prop:2}.
\begin{proof}[Proof of Proposition \ref{prop:2}]
    For $N\to\infty$ along even integers, we have
    \begin{align*}
        \pi_A(N) = \sum_{\substack{n=2 \\ n \text{ even}}}^{N} \abs{P_n} & \sim  2 \sum_{\substack{n=2 \\ n \text{ even}}}^{N} \frac{A^n}{n} \sim  2\frac{A^2}{A^2-1}\frac{A^N}{N}. \qedhere
    \end{align*}
\end{proof}

\section{Proof of Main Central Limit Theorem}\label{sec:mainproof}

To study the limiting distribution, we adopt a standard method based on characteristic functions. \\

A function $f: \N \longrightarrow \R$ is said to have \textit{limiting distribution} $w(x)$ if 
\begin{equation*}
    \lim_{N\rightarrow \infty} \frac{\abs{\{n\leq N: f(n) < x\}}}{N} = w(x).
\end{equation*}

By Lévy's Continuity Theorem, this happens exactly if 
\begin{equation*}
    \phi(t) \vcentcolon = \lim_{N\rightarrow \infty} \frac{1}{N} \sum_{n\leq N} e^{itf(n)} 
\end{equation*}

exists for all $t\in\R$ and is continuous at $0$, and in that case $\phi(t)$ is the characteristic function of $w(x)$, i.e.

\begin{equation*}
    \phi(t) = \int_{\R} e^{itx} \,dw(x).
\end{equation*}

Therefore, we may rewrite Theorem \ref{thm:mainp} as follows:

\begin{thm}\label{thm:2}
For $\sigma_p^2= \frac{A^2-1}{12}$ and $N\to \infty$ along even integers, we have

\begin{equation}\label{eq:expsum}
    \lim_{N\rightarrow \infty}\frac{1}{\pi_A(N)}\sum_{\al\in \Pi_A(N)} e^{\frac{it\Psi(\al)}{\sqrt{\ell_p(\al)}}} = e^{-\frac{1}{2}\sigma_p^2 t^2}
\end{equation}

for any $t\in\R$. Equivalently, \begin{equation}
    S_t(N)\vcentcolon= \sum_{\al\in \Pi_A(N)} e^{\frac{it\Psi(\al)}{\sqrt{\ell_p(\al)}}} \sim \pi_A(N)\cdot e^{-\frac{1}{2}\sigma^2 t^2}.
\end{equation}
\end{thm}

Our strategy is to first rewrite the exponential sum $S_t(N)$ as a sum over all sequences of length $n\leq N$ by showing that the contribution from the non-primitive words is negligible. \\

We start by writing $$S_t(N) = \sum_{\al\in \Pi_A(N)} e^{\frac{it\Psi(\al)}{\sqrt{\ell_p(\al)}}} = \sum_{\substack{n=2\\ \text{even}}}^{N} \sum_{\al\in P_n} e^{\frac{it\Psi(\al)}{\sqrt{n}}},$$ 

and since for each $\al\in P_n$, there are $\frac{n}{2}$ corresponding primitive sequences of length $n$ (via even cyclic shifts), we may continue 
\begin{equation} \label{eq:expsumsplitted}
    S_t(N) = \sum_{\substack{n=2\\ \text{even}}}^{N} \sum_{\al\in P_n} e^{\frac{it\Psi(\al)}{\sqrt{n}}}
    =  2 \sum_{\substack{n=2 \\ n \text{ even}}}^{N} \frac{1}{n} \sum_{\substack{\al\in[A]^n \\ \text{primitive}}} e^{\frac{it\Psi(\alpha)}{\sqrt{n}}}.
\end{equation}

We can now replace the inner sum over primitive sequences by the sum over all $\al\in[A]^n$. 

\begin{lemma}\label{lem:replaceprimbyall}
    $$\abs{\sum_{\al\in [A]^n} e^{\frac{it\Psi(\al)}{\sqrt{n}}} - \sum_{\substack{\al\in[A]^n \\ \textnormal{primitive}}} e^{\frac{it\Psi(\al)}{\sqrt{n}}}} = \bigO(A^{\frac{n}{2}})$$
\end{lemma}

\begin{proof}
    Follows from Corollary \ref{cor:fewnonprim} and the trivial bound.
\end{proof}

Hence, we can continue writing \eqref{eq:expsumsplitted} as
\begin{align*}
    S_t(N) & = 2 \sum_{\substack{n=2 \\ n \text{ even}}}^{N} \frac{1}{n} \sum_{\al\in[A]^n} e^{\frac{it\Psi(\alpha)}{\sqrt{n}}} + \bigO\left( \sum_{n=2 ~ \textrm{even}}^N \frac{A^{n/2}}{n}\right) \\
    & = 2 \sum_{\substack{n=2 \\ n \text{ even}}}^{N} \frac{1}{n} \sum_{\al\in[A]^n} e^{\frac{it\Psi(\alpha)}{\sqrt{n}}} + \bigO\left( \frac{A^{N/2}}{N}\right),
\end{align*}
where we use Lemma \ref{lem:sum} to control the error term. \\

At this point, we are nearly done. Given the statement of Theorem \ref{thm:2}, the error term $\bigO(\frac{A^{N/2}}{N})$ coming from the non-primitive words is negligible, and it remains to apply a standard Central Limit Theorem argument for alternating sums to finish the argument. In fact, all we need is the following lemma. 

\begin{lemma}\label{lem:probabilistic}
    Let $\sigma_p^2=\frac{A^2-1}{12}$. For all $t\in\R$ and as $n\to\infty$, we have $$\sum_{\al\in[A]^n}e^{\frac{it\Psi(\al)}{\sqrt{n}}} \sim A^n e^{-\frac{1}{2}\sigma_p^2 t^2}.$$
\end{lemma} 

With the lemma in hand, we may write
\begin{equation*}
    S_t(N) \sim 2 ~ e^{-\frac{1}{2}\sigma_p^2 t^2}\sum_{\substack{n=2 \\ n \text{ even}}}^{N} \frac{A^n}{n}  \sim ~ \pi_A(N) \cdot e^{-\frac{1}{2}\sigma_p^2 t^2}.
\end{equation*}

Therefore, it suffices to show Lemma \ref{lem:probabilistic} to conclude the proof of Theorem \ref{thm:2}.

\begin{thm}[Classical CLT]
    Let $(X_k)_{k\geq1}$ be a sequence of i.i.d. random variables, each with finite expectation $\mu$ and variance $\sigma^2>0$. Then $$\frac{\sum_{k=1}^n X_k-n\mu}{\sqrt{n}} \overset{d}{\longrightarrow} \mathcal{N}(0,\sigma^2).$$
\end{thm}

To apply this result, we let $(Y_k)_{k\geq1}$ be a sequence of i.i.d. random variables, each uniformly distributed on $\{1,\dots,A\}$, and then define $X_k = Y_{2k-1}-Y_{2k}$. The sequence $(X_k)_{k\geq1}$ is then i.i.d. and centered, that is, each $X_k$ has zero expectation.

\begin{lemma}
    The expectation and variance of $X_k$ are given by $\E[X_k] = 0$ and $\sigma^2\vcentcolon=\V(X_k) =\frac{A^2-1}{6}$.
\end{lemma}

\begin{proof}
     By independence, the variance of the $X_k$'s is computed via \begin{align*}
         \sigma^2 = \V(X)= 2\V(Y) & = 2\left( \E[Y^2]-\E[Y]^2 \right) \\
         & = 2\left(\frac{\sum_{j=1}^A j^2}{A} - \left(\frac{\sum_{j=1}^A j}{A}\right)^2\right) \\
         & = 2\left(\frac{(A+1)(2A+1)}{6}-\left(\frac{A+1}{2}\right)^2\right) = \frac{A^2-1}{6}. \qedhere
     \end{align*}
\end{proof}

It follows directly that $$\frac{\sum_{k=1}^{n/2}X_k}{\sqrt{n/2}} \overset{d}{\longrightarrow} \mathcal{N}\left(0,\sigma^2\right).$$ Equivalently, writing $\sigma_p^2 \vcentcolon = \frac{\sigma^2}{2}=\frac{A^2-1}{12}$, we get $$\frac{\sum_{k=1}^{n/2}X_k}{\sqrt{n}} \overset{d}{\longrightarrow} \mathcal{N}\left(0,\sigma_p^2\right).$$

We can express this identity in terms of characteristic functions: $$\lim_{n\rightarrow\infty}\E\left[e^{it\frac{\sum_{k=1}^{n/2} X_k}{\sqrt{n}}}\right]=e^{-\frac{1}{2}\sigma_p^2 t^2}.$$

Writing out the left-hand side yields $$\lim_{n\rightarrow\infty}\frac{1}{A^n}\sum_{\al\in[A]^n}e^{\frac{it\Psi(\al)}{\sqrt{n}}}=e^{-\frac{1}{2}\sigma_p^2 t^2},$$ and hence it follows that
\begin{equation}
    \sum_{\al\in[A]^n}e^{\frac{it\Psi(\al)}{\sqrt{n}}} \sim A^n e^{-\frac{1}{2}\sigma_p^2 t^2},
\end{equation}
which proves Lemma \ref{lem:probabilistic} and thus concludes the proof of Theorem \ref{thm:mainp}.

\section{Fixed Period Length Estimates}\label{sec:fixed} 

We first establish Berry--Esseen and local limit estimates for fixed period length $n$, and then combine them using the distribution of $\ell_p$ on $\Pi_A(N)$.

\subsection*{Berry--Esseen}

The main tool in proving a Berry--Esseen result for a single $P_n$ is \emph{Esseen's inequality}, which allows us to translate control of characteristic functions into control of cumulative distribution functions. Our presentation follows \cite[Chapter 7.6.2]{Gut2013}. 

\begin{lemma}[Esseen's Inequality]
    Let $U$ and $V$ be random variables with cumulative distribution functions $F_U$ and $F_V$ and characteristic functions $\phi_U$ and $\phi_V$. If $\sup_xF'_V(x) < \infty$, then there exists a constant $C_1>0$ such that for all $T>0$, we have
    $$\sup_{x\in\R}\abs{F_U(x)-F_V(x)} \leq \frac{1}{\pi}\int_{-T}^T\frac{\abs{\phi_U(t)-\phi_V(t)}}{\abs{t}} dt + \frac{C_1}{T}.$$
\end{lemma}

We will first use this inequality to obtain a Berry--Esseen error for each individual set $P_n$ and then combine these estimates to get a global error on $\Pi_A(N) = \cup_{n\leq N} P_n$. \\

To obtain a Berry--Esseen error for each individual set $P_n$, let $\phi_{n}$ be the characteristic function of the random variable $$\frac{\Psi(\cdot)}{\sqrt{\ell_p(\cdot)}}$$ on $\Pi_A(N)$ restricted to $P_n$. In explicit terms, we have 
\begin{equation}\label{eq:phi_n}
    \phi_n(t) = \frac{1}{\#P_n} \sum_{\al\in P_n} e^{it\frac{\Psi(\al)}{\sqrt{n}}}.
\end{equation}

By \eqref{eq:f}, we have $$\frac{1}{\#P_n} = \frac{n}{2A^n}\cdot\frac{1}{1+\bigO\left(\frac{1}{A^{n/2}}\right)} = \frac{n}{2A^n}\left( 1+ \bigO\left( \frac{1}{A^{n/2}}\right) \right).$$

As usual, we can write \begin{align*}
    \sum_{\al\in P_n} e^{it\frac{\Psi(\al)}{\sqrt{n}}} = \frac{2}{n}\sum_{\substack{\al\in[A]^n\\ \textnormal{primitive}}}e^{it\frac{\Psi(\al)}{\sqrt{n}}} = \frac{2}{n}\sum_{\al\in[A]^n}e^{it\frac{\Psi(\al)}{\sqrt{n}}} + \bigO\left(\frac{A^{n/2}}{n}\right).
\end{align*}

Putting it together, we obtain 
\begin{equation}\label{eq:phiphitilde}
\begin{aligned}
\phi_n(t)
&= \frac{1}{\#P_n} \sum_{\alpha\in P_n} e^{it\Psi(\alpha)/\sqrt{n}}  \\
&= \frac{1}{A^n}\left( 1 + \bigO(A^{-n/2}) \right)
   \left( \sum_{\alpha\in[A]^n} e^{it\Psi(\alpha)/\sqrt{n}} + \bigO(A^{n/2}) \right) \\
&= \underbrace{\frac{1}{A^n}\sum_{\alpha\in[A]^n} e^{it\Psi(\alpha)/\sqrt{n}}}_{:=\widetilde{\phi}_n(t)}
   + \bigO(A^{-n/2}).
\end{aligned}
\end{equation}

where $\widetilde{\phi}_n$ is the characteristic function of the suitably normalized alternating sum of uniform i.i.d. random variables on the finite set $[A]$. For such random variables, the third moment is finite. (In fact, all moments are!) It then follows from \cite[Lemma 7.6.2]{Gut2013} that there exists some constant $C_2=C_2(A)$ independent of $n$ such that 
\begin{equation}\label{eq:cfcontrol}
    \abs{\widetilde{\phi}_n(t)-e^{-\frac{1}{2}\sigma^2 t^2}} \leq \frac{C_2}{\sqrt{n}}\abs{t}^3 e^{-\frac{1}{3}t^2}
\end{equation}

for $\abs{t} \leq C_2\sqrt{n}$. \\

We now apply Esseen's inequality to $\phi_n$ and $V=\mathcal{N}(0,\sigma_p^2)$.
For this, we pick $T=T(n)=C_2\sqrt{n}$ such that $\frac{C_1}{T}\ll\frac{1}{\sqrt{n}}$. We also define $\delta=\delta(n)=\frac{1}{n^{1/6}}$ and split the relevant integral into two regions: $$\int_{\abs{t}\leq\delta} + \int_{\delta\leq\abs{t}\leq T} =\vcentcolon I_1 + I_2.$$

Let us first treat $I_1$. For small $\abs{t}$, we have $$e^{-\frac{1}{2}\sigma^2 t^2} = 1 -\frac{1}{2}\sigma^2 t^2 + \bigO(\abs{t}^4)$$ as well as $$\phi_n(t) = 1  -\frac{1}{2}\sigma^2 t^2 + \bigO(\abs{t}^3)$$ up to an exponentially small, negligible error coming from removing the non-primitive words. Therefore, we get
$$I_1 = \int_{\abs{t}\leq\delta}\frac{\abs{\phi_U(t)-\phi_V(t)}}{\abs{t}} dt \ll \int_{\abs{t}\leq\delta} t^2 dt \asymp \delta^3=\frac{1}{\sqrt{n}}.$$

For $I_2$, we plug in \eqref{eq:cfcontrol}, which gives us
\begin{align*}
    I_2 & = \int_{\delta\leq\abs{t}\leq T} \frac{\abs{\phi_U(t)-\phi_V(t)}}{\abs{t}} dt \\
    &\ll \frac{1}{\sqrt{n}}\int_{\R}t^2e^{-\frac{1}{3}t^2}dt + \frac{1}{A^{n/2}}\int_{\delta\leq\abs{t}\leq T} \frac{dt}{t} \ll \frac{1}{\sqrt{n}} +  \frac{\log(T/\delta)}{A^{n/2}} \ll \frac{1}{\sqrt{n}} +  \frac{\log(n)}{A^{n/2}}.
\end{align*}

Therefore, we have proved the following:

\begin{prop}\label{prop:BEn}
    For $n$ even, we have 
    $$\sup_{x\in\R} \abs{\Prob\left( \frac{\Psi}{\sqrt{\ell_p}}\leq x ~\middle|~ \ell_p = n \right)- \Phi(x)} \ll \frac{1}{\sqrt{n}}.$$
\end{prop}

\subsection*{Local Limit Theorem}

Next, we prove a local limit theorem for $P_n$. 

\begin{prop}\label{prop:localllt}
    For any $k\in\Z$ and $n\in2\N$, we have
    $$\Prob(\Psi=k \mid \ell_p=n) = \frac{1}{\sqrt{2\pi\sigma^2 n}}\exp\left(-\frac{k^2}{2\sigma^2 n}\right) + o\left(\frac{1}{\sqrt{n}}\right).$$
\end{prop}

\begin{proof}
    Using Fourier inversion and the change of variables $t=s\sqrt{n}$ yields 
    \begin{align*}
        \Prob(\Psi=k \mid \ell_p=n) & = \frac{1}{2\pi}\int_{-\pi}^{\pi}\E[e^{is\Psi}\mid \ell_p=n]\cdot e^{-isk} \,ds \\
        & = \frac{1}{2\pi\sqrt{n}}\int_{-\pi\sqrt{n}}^{\pi\sqrt{n}}\phi_n(t) e^{-itk/\sqrt{n}} \,dt,
    \end{align*}

with $\phi_n(t)$ as in \eqref{eq:phi_n}. By $\eqref{eq:phiphitilde}$, we can further estimate
\begin{align*}\frac{1}{2\pi\sqrt{n}}\int_{-\pi\sqrt{n}}^{\pi\sqrt{n}}\phi_n(t) e^{-itk/\sqrt{n}} \,dt & = \frac{1}{2\pi\sqrt{n}}\int_{-\pi\sqrt{n}}^{\pi\sqrt{n}}\widetilde{\phi}_n(t) e^{-itk/\sqrt{n}} \,dt + \bigO\left( \frac{1}{2\pi\sqrt{n}}\int_{-\pi\sqrt{n}}^{\pi\sqrt{n}} A^{-n/2}\,dt\right) \\
& = \frac{1}{2\pi\sqrt{n}}\int_{-\pi\sqrt{n}}^{\pi\sqrt{n}}\widetilde{\phi}_n(t) e^{-itk/\sqrt{n}} \,dt + \bigO\left( A^{-n/2}\right).
\end{align*}

The last integral is equal to the probability that the i.i.d. alternating sum is equal to $k$, and thus we can apply a classical Local Limit Theorem --- see e.g. \cite[Section 7.5]{Gut2013} --- and the error term $\bigO(A^{-n/2})$ becomes negligible. 
\end{proof}

\section{Estimates for Varying Period Lengths}\label{sec:varying}

We will now combine our estimates from the previous section for varying $n\leq N$.

\subsection*{Berry--Esseen}
The global Berry--Esseen result follows from Proposition \ref{prop:BEn}.

\begin{proof}[Proof of Theorem \ref{thm:BE}]
    For any even integer $N$ and $x\in\R$, we write $$\Prob_N\left(\frac{\Psi}{\sqrt{\ell_p}} \leq x\right) = \sum_{\substack{n=2\\ \textnormal{even}}}^N \Prob_N(\ell_p = n)\cdot\Prob\left( \frac{\Psi}{\sqrt{\ell_p}}\leq x ~\middle|~ \ell_p = n \right).$$

Subtracting $\Phi(x)$, taking the supremum over all $x\in\R$ and using Proposition, we get
\begin{align*}\sup_{x\in\R}\abs{\Prob_N\left(\frac{\Psi}{\sqrt{\ell_p}} \leq x\right)-\Phi(x)} & \leq \sum_{\substack{n=2\\ \textnormal{even}}}^N \Prob_N(\ell_p = n) \cdot \sup_{x\in\R}\abs{\Prob\left( \frac{\Psi}{\sqrt{\ell_p}}\leq x ~\middle|~ \ell_p = n \right)-\Phi(x)} \\
& \ll \sum_{\substack{n=2\\ \textnormal{even}}}^N \frac{\Prob_N(\ell_p = n)}{\sqrt{n}} = \E_N\left[\frac{1}{\sqrt{\ell_p}}\right].
\end{align*}

To estimate the last expression, we use \eqref{eq:f} and Proposition \ref{prop:2}:

$$\E_N\left[\frac{1}{\sqrt{\ell_p}}\right] = \frac{1}{\pi_A(N)} \sum_{\substack{n=2\\ \textnormal{even}}}^N \frac{1}{\sqrt{n}}\left(\frac{2A^n}{n}+ \bigO\left(\frac{A^{n/2}}{n}\right)\right) \ll \frac{N}{A^N}\sum_{\substack{n=2\\ \textnormal{even}}}^N \frac{A^n}{n^{3/2}} \asymp \frac{N}{A^N}\frac{A^N}{N^{3/2}} = \frac{1}{\sqrt{N}}.$$

The estimate for the last sum follows from Stolz-Cesàro, similar to Lemma \ref{lem:sum}.
\end{proof}

\subsection*{Local Limit Theorem}
Similarly, the Local Limit Theorem \ref{thm:llt} for the union $\Pi_A(N)$ follows from the corresponding Proposition \ref{prop:localllt} for a fixed period length.

\begin{proof}[Proof of Theorem \ref{thm:llt}]

Let $M=M(N)=\sqrt{N}$, and split the even integers up to $N$ into a small and large regime:
\begin{align*}
    \mathcal{S}=\mathcal{S}(N) & = 2\Z \,\cap\, [1,N-M]  \\
    \mathcal{L}=\mathcal{L}(N) & = 2\Z \,\cap\, (N-M,N].
\end{align*}
Then we have 
\begin{equation}\label{eq:largepl}
    \Prob_N(\ell_p \in \mathcal{L}) = 1+o(1),
\end{equation} because its 
$$\Prob_N(\ell_p\in\mathcal{S}) = \frac{1}{\pi_A(N)}\sum_{n\in\mathcal{S}}\#P_n \ll \frac{N}{A^N}\sum_{n\in\mathcal{S}} \frac{A^n}{n} \ll \frac{N}{A^N}\sum_{n\in\mathcal{S}}A^{N-M} \ll N^2A^{-M}.$$

This observation invites us to split
\begin{equation*}
     \Prob_N(\Psi=k) = \sum_{\substack{n=2\\ \textnormal{even}}}^N \Prob_N(\ell_p = n)\Prob(\Psi=k \mid \ell_p = n)
\end{equation*} 
into two sums
\begin{equation}\label{eq:bigeqllt}
    \Prob_N(\Psi=k) = \sum_{n\in\mathcal{S}} \Prob_N(\ell_p = n)\Prob(\Psi=k \mid \ell_p = n) + \sum_{n\in\mathcal{L}} \Prob_N(\ell_p = n)\Prob(\Psi=k \mid \ell_p = n).
\end{equation}

To deal with the first sum, we use  
$$\sum_{n\in\mathcal{S}} \Prob_N (\ell_p = n) \Prob(\Psi=k \mid \ell_p = n) \leq \sum_{n\in\mathcal{S}} \Prob_N(\ell_p=n) = \Prob_N(\ell_p\in \mathcal{S}) \ll N^2A^{-\sqrt{N}}.$$
Therefore, the first sum of \eqref{eq:bigeqllt} is exponentially small in $\sqrt{N}$.\\

The second sum of \eqref{eq:bigeqllt} deals with $n\in\mathcal{L}$. It is here that we make use of the Local Limit Theorem we have established for each individual $n$.
\begin{align*}
    & \sum_{n\in\mathcal{L}} \Prob_N(\ell_p = n)\Prob(\Psi=k \mid \ell_p = n) \\
    & =  \sum_{n\in\mathcal{L}} \Prob_N(\ell_p = n)\left(\frac{1}{\sqrt{2\pi\sigma^2 n}}\exp\left(-\frac{k^2}{2\sigma^2 n}\right) + r(n) \right) \\
    & =  \sum_{n\in\mathcal{L}} \Prob_N(\ell_p = n)\left(\frac{1}{\sqrt{2\pi\sigma^2 n}}\exp\left(-\frac{k^2}{2\sigma^2 n}\right)\right) \, + \, \sum_{n\in\mathcal{L}} \Prob_N(\ell_p = n)\cdot r(n) \\
    & =\vcentcolon  \Sigma_1(N) + \Sigma_2(N)
\end{align*}

for some $r(n)=o\left(\frac{1}{\sqrt{n}}\right)$. \\

Note that $\Sigma_2(N)=o\left(\frac{1}{\sqrt{N}}\right)$. Indeed, pick any $\varepsilon>0$. There exists $n_0$ such that for any $n>n_0$, we have $\abs{r(n)} < \varepsilon/\sqrt{n}$. Now let $N$ be large enough such that $N-M>n_0$. Then we have 
\begin{align*}
    \sqrt{N} \abs{\Sigma_2(N)} & = \sqrt{N}\cdot\sum_{n\in\mathcal{L}} \Prob_N(\ell_p = n)\abs{r(n)} < \ \varepsilon\sqrt{\frac{N}{N-M}}\sum_{n\in\mathcal{L}} \Prob_N(\ell_p = n) \leq  \varepsilon \sqrt{\frac{N}{N-M}}  < 2\varepsilon.
\end{align*}

We will now handle $\Sigma_1(N)$. Since we deal with $n\in\mathcal{L}$, we have $n=N+\bigO(M)$, and thus
\begin{equation*}
    \frac{1}{\sqrt{n}} = \frac{1}{\sqrt{N}}\left(1+\bigO(M/N)\right)^{-1/2} = \frac{1}{\sqrt{N}}\left(1+\bigO(M/N)\right).
\end{equation*}

Moreover, we have 
\begin{equation*}
    \frac{1}{n} = \frac{1}{N}\cdot\frac{1}{1+\bigO(M/N)}= \frac{1}{N}(1+\bigO(M/N)) = \frac{1}{N} + \bigO\left(\frac{M}{N^2}\right),
\end{equation*}
and since we're assuming that $k\ll\sqrt{N}$, we get
\begin{align*}
    \exp\left(-\frac{k^2}{2\sigma^2n}\right) & = \exp\left(-\frac{k^2}{2\sigma^2}\cdot\left(\frac{1}{N} + \bigO\left(\frac{M}{N^2}\right)\right)\right) = \exp\left(-\frac{k^2}{2\sigma^2N}\right)\exp\left(\bigO\left(\frac{k^2M}{N^2}\right)\right) \\
    & = \exp\left(-\frac{k^2}{2\sigma^2N}\right) \cdot \left(1+\bigO(M/N)\right).
\end{align*}

This allows us to replace the Gaussian terms for each individual $n\in\mathcal{L}$ by the Gaussian term for $N$ and a uniform error:
\begin{align*}
    \frac{1}{\sqrt{2\pi\sigma^2 n}}\exp\left(-\frac{k^2}{2\sigma^2 n}\right) & = \frac{1}{\sqrt{2\pi\sigma^2N}}\left(1+\bigO(M/N)\right)\cdot\exp\left(-\frac{k^2}{2\sigma^2N}\right) \cdot \left(1+\bigO(M/N)\right) \\
    & = \frac{1}{\sqrt{2\pi\sigma^2N}}\cdot\exp\left(-\frac{k^2}{2\sigma^2N}\right) \cdot \left(1+\bigO(M/N)\right).
\end{align*}

We can thus estimate $\Sigma_1(N)$ by 
\begin{align*}
    \Sigma_1(N) & = \sum_{n\in\mathcal{L}} \Prob_N(\ell_p = n)\left(\frac{1}{\sqrt{2\pi\sigma^2 n}}\exp\left(-\frac{k^2}{2\sigma^2 n}\right)\right) \\
    & = \frac{1}{\sqrt{2\pi\sigma^2N}}\cdot\exp\left(-\frac{k^2}{2\sigma^2N}\right) \cdot \left(1+\bigO(M/N)\right)\cdot \sum_{n\in\mathcal{L}}\Prob_N(\ell_p = n) \\
    & = \frac{1}{\sqrt{2\pi\sigma^2N}}\cdot\exp\left(-\frac{k^2}{2\sigma^2N}\right) \cdot \left(1+\bigO(M/N)\right)\cdot (1+o(1))\\
    & = \frac{1}{\sqrt{2\pi\sigma^2N}}\cdot\exp\left(-\frac{k^2}{2\sigma^2N}\right) + o\left(\frac{1}{\sqrt{N}}\right)
\end{align*}

Finally, we can collect all our estimates to evaluate \eqref{eq:bigeqllt}:
\begin{align*}
    \Prob_N(\Psi=k) & = \sum_{n\in\mathcal{S}} \Prob_N(\ell_p = n)\Prob(\Psi=k \mid \ell_p = n) + \sum_{n\in\mathcal{L}} \Prob_N(\ell_p = n)\Prob(\Psi=k \mid \ell_p = n) \\
    & \ll N^2 A^{-\sqrt{N}}  + \Sigma_1(N) + \Sigma_2(N)\\
    & = N^2 A^{-\sqrt{N}} + \frac{1}{\sqrt{2\pi\sigma^2N}}\cdot\exp\left(-\frac{k^2}{2\sigma^2N}\right) + o\left(\frac{1}{\sqrt{N}}\right) + o\left(\frac{1}{\sqrt{N}}\right) \\
    & = \frac{1}{\sqrt{2\pi\sigma^2N}}\cdot\exp\left(-\frac{k^2}{2\sigma^2N}\right) + o\left(\frac{1}{\sqrt{N}}\right). \qedhere
\end{align*}
\end{proof}

\section{Comparison of Length Functions}\label{sec:comparison}

In this last section, we prove the Comparison Theorem \ref{thm:comparison} and deduce Theorem \ref{thm:otherlengths} as an immediate consequence. \\

Indeed, let $\ell_*$ be any standard length, and consider the random variable 
\begin{equation*}
    Z^*_N=\frac{\Psi}{\sqrt{\ell_*}}
\end{equation*}
on $\Pi_A(N)$, which we can write as 
\begin{equation*}
    Z^*_N 
= \underbrace{\frac{\Psi}{\sqrt{\ell_p}}}_{= W_N} \sqrt{\frac{\ell_p}{\ell_*}},
\end{equation*}
with $W_N$ converging in distribution to a centered Gaussian; see Theorem \ref{thm:mainp}.\\

We want to apply Slutsky's Theorem, see \cite[Thm 5.11.4]{Gut2013}.

\begin{prop}[Slutsky's Theorem]
    Let $(A_N)_N$ and $(B_N)_N$ be sequences of random variables. If 
    $$A_N \overset{\Prob} {\longrightarrow} \hat{c}, \hspace{1cm} B_N \overset{d}{\longrightarrow} B$$
    as $N \to \infty$, where $\hat{c}$ is a constant, then 
    $$A_N\cdot B_N \overset{d}{\longrightarrow} \hat{c} B.$$
\end{prop}

Together with the Continuous Mapping Theorem \cite[Thm 5.10.4]{Gut2013}, this reduces Theorem $\ref{thm:otherlengths}$ to the Comparison Theorem \ref{thm:comparison}. We prove this theorem by considering its three cases separately.

\subsection*{Maximal length}
By \eqref{eq:largepl}, almost all geodesics have ``large'' period length in the sense that 
\begin{equation*}
    \Prob_N\left(1-\frac{1}{\sqrt{N}} \le \frac{\ell_p}{N} \le 1\right)\to 1
\end{equation*}
as $N\to\infty$. Equivalently, the ratio $\ell_p/N$ converges in probability to $1$, which is precisely the desired comparison in this case.

\subsection*{Geometric Length}
The geometric length of a closed geodesic on the modular surface is typically expressed in terms of the fundamental unit of the imaginary quadratic field $\mathbb{Q}(\sqrt{D})$, where $D$ is the discriminant of the associated class of binary quadratic forms. (In particular, all closed geodesics sharing the same discriminant have the same length.) \\

However, one may also compute the geometric length solely in terms of the partial quotients of the associated necklace. 

\begin{prop}
    The geometric length of a closed geodesic with minimal even periodic part $\alpha=(a_1,\dots,a_n)$ is given by $$\ell_g(\alpha) = 2\sum_{j=1}^n \log(T^j \overline{\al}),$$ where $T: \R_{>0}\setminus\Z \to \R_{>0}$ given by $T(x) = \frac{1}{x-\floor{x}}$ denotes the Gauss map, and $\overline{\al}$ is the quadratic irrational defined by extending the continued fraction expansion of $\alpha$ periodically.
\end{prop}

\begin{proof}
    See \cite[Section 3.2]{Series1985}. 
\end{proof}

Kelmer \cite{Kelmer2012} used this expression for the geometric length to prove that closed geodesics of a fixed winding number equidistribute on the modular surface when ordered by $\ell_g$.

\begin{rem}
    For $x>0$ with continued fraction $x=[a_1,a_2,a_3,\dots]$, we have $T(x) = [a_2,a_3,\dots]$.
\end{rem}

\begin{prop}
    As $N\to \infty$ along even integers, the random variable $C_N$ defined on $\Pi_A(N)$ via $$ C_N(\al) \vcentcolon= \frac{\ell_g(\al)}{\ell_p(\al)} = \frac{2}{n}\sum_{j=1}^n \log(T^j \overline{\al})$$ converges in probability to a constant. 
\end{prop}

Because we are dealing with uniformly bounded partial quotients, it is clear that $C_N(\cdot)$ is also uniformly bounded. However, that is not enough to deduce convergence to a constant. \\

Instead, we have to prove some kind of law of large numbers or Birkhoff Ergodic Theorem for the (countable!) subset of periodic sequences in $\{1,\dots,A\}^\N$, where the summands are dependent. Similar results have been obtained for periodic orbits of Axiom A flows \cite{Lalley1987}.\\

We can reduce the problem of studying the convergence of $C_N$ to that of a simpler random variable. 

\begin{lemma}
    If the random variable \begin{align*}
        [A]^N & \longrightarrow \R,\\
        \alpha & \mapsto \frac{2}{N}\sum_{j=1}^N \log(T^j \overline{\al})
    \end{align*} 
    converges in probability to a constant $\hat{c}$, then so does $C_N$.
\end{lemma}

\begin{proof}
    For any $N$ even, define the following sets: 
\begin{align*}
    \Omega_{N}^{(1)} & = \{\alpha = (a_1,\dots,a_N): a_i\leq A\} = [A]^N,\\
    \Omega_{N}^{(2)} & = \{\alpha = (a_1,\dots,a_N): a_i\leq A \text{ and } \al \text{ primitive}\},\\
    \Omega_{N}^{(3)} & = \Omega_{N}^{(2)}/\sim \\
    \Omega_{N}^{(4)} & = \bigsqcup_{n\leq N \text{ even}} \Omega_{n}^{(3)}
\end{align*}

Note that $\Omega_{N}^{(3)} = P_N$ and $\Omega_{N}^{(4)} = \Pi_A(N)$. On each of the $\Omega_{N}^{(i)}$, there exists a natural analogue $C_N^{(i)}$ of the random variable $C_N=C_N^{(4)}$ on $\Omega_{N}^{(4)}$. \\

Now assume that $C_N^{(1)}$, that is, the random variable in the statement of the lemma, converges in probability to $\hat{c}$. In three steps, we will show how to upgrade this to the desired convergence of $C_N^{(4)}$ to $\hat{c}$ on $\Omega_{N}^{(4)}$. \\

\textit{Step 1:} By Proposition \ref{prop:1}, the subset of non-primitive sequences forms a set of vanishing measure as $N\to\infty$. \\

\textit{Step 2:} Every necklace in $\Omega_{N}^{(3)}$ has exactly $\frac{N}{2}$ corresponding sequences in $\Omega_{N}^{(2)}$. \\

\textit{Step 3:} Fix any $\eps>0$. We want to show that $$\Prob_N\Bigl(\abs{C_N^{(4)}-\hat{c}}>\eps\Bigr)\longrightarrow 0.$$ 

Note that $C_N^{(4)}$ restricted to each of its components $\Omega_{n}^{(3)}$ coincides with $C_n^{(3)}$. So for any $\eps>0$, we have $$\Prob_N\Bigl(\abs{C_N^{(4)}-\hat{c}}>\eps\Bigr)=\sum_{n\leq N \text{ even}} \Prob_N(\ell_p=n)\Prob_n\Bigl(\abs{C_n^{(3)}-\hat{c}}>\eps\Bigr).$$
By assumption, $\Prob_n\Bigl(\abs{C_n^{(3)}-\hat{c}}>\eps\Bigr) \longrightarrow 0$ as $n\to\infty.$ This implies that for any $\delta>0$ there exists some $N_0=N_0(\delta)$ such that 
\begin{equation}\label{eq:split}
    \forall n > N_0: \hspace{0.5cm} \Prob\Bigl(\abs{C_n^{(3)}-\hat{c}}>\eps\Bigr) < \delta. 
\end{equation}
For large $N$, we split the sum into
\begin{align*}
\Prob_N\Bigl(\abs{C_N^{(4)}-\hat{c}}>\eps\Bigr) & = \sum_{n\leq N \text{ even}} \Prob_N(\ell_p=n)\Prob_n\Bigl(\abs{C_n^{(3)}-\hat{c}}>\eps\Bigr) \\
& = \sum_{n \leq N_0 \text{ even}} \Prob_N(\ell_p=n)\Prob_n\Bigl(\abs{C_n^{(3)}-\hat{c}}>\eps\Bigr) \\
& + \sum_{N_0 < n \leq N\text{ even}} \Prob_N(\ell_p=n)\Prob_n\Bigl(\abs{C_n^{(3)}-\hat{c}}>\eps\Bigr).
\end{align*}

The first sum vanishes as $N\to\infty$ since $\Prob_N(\ell=n) \asymp \frac{N}{n}A^{n-N} \longrightarrow 0$. Moreover, the second sum can easily be bounded from above by $\delta$ using $\eqref{eq:split}$, and the result follows. 
\end{proof}

Therefore, it remains to show:
\begin{lemma}\label{lem:lastlem}
    The random variable \begin{equation*}
    C_N: [A]^N\ni\alpha \mapsto \frac{2}{N}\sum_{j=1}^N \log(T^j \overline{\al})
\end{equation*} converges in probability to a constant.
\end{lemma}

\begin{proof}

The proof relies on the following approximation scheme:

\begin{figure}[H]
    \centering
    \begin{tikzcd}[row sep=6em, column sep=10em]
        C_{N,k} \arrow[d, "k\to\infty" left, thick] \arrow[r, "N\to\infty" above, thick] & c_k \arrow[d, "k\to\infty" right, thick] \\
        C_N \arrow[r, dashed, "N\to\infty" below, thick] & \hat{c}
    \end{tikzcd}
    \caption{Proof idea for Lemma \ref{lem:lastlem}}
    \label{fig:enter-label}
\end{figure}

The random variables $C_{N,k}$ are obtained from $C_N$ by truncating the continued fraction expansion after $k$ partial quotients. We will show that these random variables converge uniformly to $C_N$ as $k\to\infty$, and in probability to a constant $c_k$ as $N\to\infty$. Moreover, we show that the sequence $(c_k)_k$ is a Cauchy sequence and thus converges to some $\hat{c}$. Combining all of this allows us to deduce that $C_N$ converges in probability to $\hat{c}$. \\

Let us now define the quantities involved in the approximation scheme. \\

For $k\in\N$ and $\beta\in\R$, let $[\beta]_k$ to be the $k$-th convergent to $\beta$, i.e. the rational approximation obtained by cutting the continued fraction expansion of $\beta$ after its first $k$ digits. Moreover, we define the following sequence of random variables: $$C_{N,k}: ~ [A]^N\ni\alpha \mapsto \frac{2}{N}\sum_{j=1}^N \log\left([T^j \overline{\al}]_k\right).$$

So if $\alpha=(a_1,\dots,a_N)$, we have for instance $$C_{N,1}(\alpha) = \frac{2}{N}\sum_{j=1}^N \log(a_j).$$

The random variables $C_{N,k}$ converge uniformly (and exponentially fast) to $C_N$ as $k\to\infty$, since 
\begin{align*}
    \abs{C_N(\al)-C_{N,k}(\al)} & = \frac{1}{N}\abs{\sum_{j=1}^N \log(T^j\overline{\al})-\log([T^j\overline{\al}]_k)} \\
    & \leq \frac{1}{N}\sum_{j=1}^N \abs{T^j\overline{\al}-[T^j\overline{\al}]_k}.
\end{align*}

Here, we used the fact that any real number $x>1$ with all partial quotients $a_i\leq A$ satisfies $x\geq [\overline{1,A}] = \frac{A+\sqrt{A^2+4A}}{2A}$, and that $\log$ is 1-Lipschitz on $(1+\eps,\infty)$ for any $\eps>0$.\\

As is well-known, the distance between an irrational number and its convergents decays exponentially fast and can be made uniform. Indeed, for any $x\in\R\setminus\Q$, we have 
\begin{equation}
    \abs{x-[x]_k}\leq \frac{1}{2^{k-1}},
\end{equation} by \cite[Lemma 1.24]{Aigner2013} and therefore $$\sup_{\substack{N > 0 \\ \alpha \in [A]^N}}\abs{C_N(\al)-C_{N,k}(\al)} \leq \frac{1}{2^{k-1}}\xrightarrow{k \to \infty} 0.$$

Now consider $k$ fixed. We have
\begin{align*}
    C_{N,k}(\alpha) & = \frac{2}{N}\sum_{j=1}^N \log([T^j \overline{\al}]_k)\\
    & = 2 \sum_{\al_k\in[A]^k}\frac{\abs{\{\text{(length $k$)-subwords of $\alpha$ equal to $\al_k$}\}}}{N}\cdot\log(\al_k). 
\end{align*}

We may set up a Markov chain, where the states are given by all words of length $k$, and we move from any given word $(a_1,\dots,a_k)$ to $(a_2,\dots,a_k,a^*)$ with probability $\frac{1}{A}$ for any $a^* \in [A]$. This Markov chain is clearly irreducible, as we can move between any two words in at most $k$ steps. Note that the ratio

\begin{equation*}
    \frac{\abs{\{\text{(length $k$)-subwords of $\alpha$ equal to $\al_k$}\}}}{N}\
\end{equation*} 

can be interpreted as the proportion of time our Markov chain spends in state $\alpha_k$ before time $N$. By the Ergodic Theorem for Markov chains \cite[Theorem 5.3]{Casarotto2007}, all these frequencies converge in probability to $\frac{1}{A^k}$ as $N\to\infty$. Since we're only dealing with finitely many summands, it follows directly that as $N\to\infty$,
$$C_{N,k}(\alpha) \overset{\Prob}{\longrightarrow} \frac{2}{A^k}\sum_{\alpha_k\in[A]^k} \log(\alpha_k) =\vcentcolon c_k.$$

Next, we claim that the sequence $(c_k)_k$ is Cauchy. Indeed, for any $c_k$ and $c_j$ with $k<j$, consider the difference $$\abs{c_k-c_j} = \abs{\frac{2}{A^k}\sum_{\alpha_k\in[A]^k} \log(\alpha_k) - \frac{2}{A^j}\sum_{\beta_j\in[A]^j} \log(\beta_j)}.$$

We make the very simple \textendash{} but crucial \textendash{} observation that any word $\beta_j$ is a concatenation $[\al_k, \gamma_{j-k}]$, where $\al_k$ and $\gamma_{j-k}$ are words of length $k$ and $j-k$, respectively. Therefore: 

\begin{align*}
    \abs{c_k-c_j} & = \abs{\frac{2}{A^k}\sum_{\alpha_k\in[A]^k} \log(\alpha_k) - \frac{2}{A^j}\sum_{\beta_j\in[A]^j} \log(\beta_j)} \\
    & = \abs{\frac{2}{A^k}\sum_{\alpha_k\in[A]^k} \log(\alpha_k)  - \frac{2}{A^j}\sum_{\alpha_k\in[A]^k} \sum_{\gamma_{j-k}\in[A]^{j-k}}\log([\al_k,\gamma_{j-k}])} \\
    & \leq \frac{2}{A^k} \sum_{\alpha_k\in[A]^k} \abs{\log(\al_k) - \frac{1}{A^{j-k}}\sum_{\gamma_{j-k}\in[A]^{j-k}}\log([\al_k,\gamma_{j-k}])} \\
    & \leq \frac{2}{A^k} \sum_{\alpha_k\in[A]^k} \frac{1}{A^{j-k}}\sum_{\gamma_{j-k}\in[A]^{j-k}}\abs{\log(\al_k) - \log([\al_k,\gamma_{j-k}])} \\
    & \leq \frac{2}{A^k} \sum_{\alpha_k\in[A]^k} \frac{1}{A^{j-k}}\sum_{\gamma_{j-k}\in[A]^{j-k}}\underbrace{ \abs{\alpha_k - [\alpha_k, \gamma_{j-k}]} }_{\leq \frac{1}{2^{k-1}}} \leq \frac{1}{2^{k-2}} \xrightarrow{k \to \infty} 0, 
\end{align*}

which implies that the sequence $(c_k)_k$ is Cauchy and converges to $$\hat{c} \vcentcolon = \lim_{k\to\infty} c_k = \lim_{k\to\infty}\frac{2}{A^k}\sum_{\alpha_k\in[A]^k} \log(\alpha_k).$$

We can now put everything together to conclude that $C_N$ converges in probability to $\hat{c}$. In other words, we want to prove that for any $\eps>0$, we have $$\Prob_N\Bigl(\abs{C_N-\hat{c}}>\eps\Bigr)\xrightarrow{N\to\infty} 0.$$ 

For any $k\geq 1$, we can write

\begin{equation}\label{eq:triangleineq}
\begin{aligned}
    & \Prob_N\Bigl(\abs{C_N - \hat{c}} > \varepsilon\Bigr)\\ 
    \leq ~ & \Prob_N\Bigl(\abs{C_N - C_{N,k}} > \tfrac{\varepsilon}{3}\Bigr)
    + \Prob_N\Bigl(\abs{C_{N,k} - c_k} > \tfrac{\varepsilon}{3}\Bigr) 
    + \mathbbm{1}_{\abs{c_k - \hat{c}} > \tfrac{\varepsilon}{3}}.
\end{aligned}
\end{equation}

Since $c_k\to\hat{c}$, there exists some $k_0$ such that for all $k>k_0$, we have $\abs{c_k-\hat{c}}<\frac{\eps}{3}$. \\

Moreover, since $C_{N,k}$ converges uniformly to $C_N$, there exists some $k_1$ such that for all $k>k_1$ and all $N\geq 1$, we have $$\Prob_N\left(\abs{C_{N}-C_{N,k}}>\frac{\eps}{3}\right) =0.$$

Now let $\delta>0$ and fix some $k\geq \max\{k_0,k_1\}$. Since $C_{N,k}$ converges in probability to $c_k$ as $N\to\infty$, there exists $N_0$ such that for all $N>N_0$, we have $$\Prob_N\left(\abs{C_{N,k}-c_k}>\frac{\eps}{3}\right) < \delta.$$

Therefore, for such $k$, we can bound the first summand of \eqref{eq:triangleineq} by $\delta$, whereas the second and third summand vanish, and thus $$\Prob_N\Bigl(\abs{C_N - \hat{c}} > \varepsilon\Bigr) < \delta.$$

Since $\delta>0$ was arbitrary, this completes the proof.
\end{proof}

Before turning to word length, we want to briefly discuss the variance $\sigma_g^2$. Since $\frac{\ell_g}{\ell_p}$ converges in probability to $\hat{c}$, it follows that $\sqrt{\frac{\ell_p}{\ell_g}}$ converges to $\frac{1}{\sqrt{\hat{c}}}$. Hence
$$\frac{\Psi}{\sqrt{\ell_g}}
= \underbrace{\frac{\Psi}{\sqrt{\ell_p}}}_{\xrightarrow{d} \mathcal{N}(0,\sigma_p^2)} 
\cdot ~ \underbrace{\sqrt{\frac{\ell_p}{\ell_g}}}_{\xrightarrow{\Prob}\frac{1}{\sqrt{\hat{c}}}} \xrightarrow{\hspace{0.2cm} d \hspace{0.2cm}} \mathcal{N}\left(0,\frac{\sigma_p^2}{\hat{c}}\right),$$
and the variance $\sigma_g^2$ is given by $$\sigma_g^2 = \frac{\sigma_p^2}{\hat{c}} = \frac{A^2-1}{12\hat{c}}.$$

\subsection*{Word Length}
The typical generators of $\G$ are the infinite order translation $T$ and the reflection $S$ of order $2$ given by $$T=\pm \begin{pmatrix}
    1 & 1 \\
    1 & 0
\end{pmatrix}, ~~ S = \pm \begin{pmatrix}
    0 & -1 \\ 1 & 0
\end{pmatrix}.$$ Alternatively, defining $$M = TS = \pm \begin{pmatrix}
    1 & -1 \\ 1 & 0
\end{pmatrix},$$ we have the group presentation $$\Gamma =  \langle S, M \mid S^2 = M^3 = 1 \rangle \cong \Z_2 * \Z_3.$$

Recalling \eqref{eq:matrixproduct}, we associate with each necklace $\al=(a_1,\dots,a_n)$ of even length $n$ its hyperbolic group element \begin{equation}\label{eq:gamma_al}
    \pm \ga_{\al} = \prod_{i=1}^n \begin{pmatrix}
    a_i & 1 \\ 1 & 0
\end{pmatrix}=\prod_{i=1}^{n/2}\Biggl(\begin{pmatrix}
    a_{2i-1} & 1 \\ 1 & 0
\end{pmatrix}\begin{pmatrix}
    a_{2i} & 1 \\ 1 & 0
\end{pmatrix}\Biggr).
\end{equation}

One checks that for all $a,b\in\Z$, we have $$\begin{pmatrix}
    a & 1 \\ 1 & 0
\end{pmatrix}\begin{pmatrix}
    b & 1 \\ 1 & 0
\end{pmatrix} = T^aST^{-b}S = (MS)^aS(MS)^{-b}S = (MS)^a(M^{-1}S)^b,$$ and hence $$\ell_w\left( \begin{pmatrix}
    a & 1 \\ 1 & 0
\end{pmatrix}\begin{pmatrix}
    b & 1 \\ 1 & 0
\end{pmatrix} \right) = 2a+2b.$$

Therefore, a necklace $\al=(a_1,\dots,a_n)$ with $\ga_{\al}$ as in \eqref{eq:gamma_al} has word length given by $$\ell_w(\al)\vcentcolon=\ell_w(\ga_{\al}) = 2\sum_{i=1}^n a_i.$$ 

Using the same arguments as in the geometric length case, it suffices to show that the random variable $$\frac{\ell_w}{\ell_p}:[A]^N\ni\al=(a_1,\dots,a_N) \mapsto \frac{2\sum_{i=1}^{N} a_i}{N} $$ converges to a constant as $N\to\infty$. In contrast to the geometric length case, this follows directly from the Law of Large Numbers due to the independence of the individual summands, and the constant is given by $$\frac{\ell_w}{\ell_p} \xrightarrow{\hspace{0.1cm} \Prob \hspace{0.1cm} } 2\frac{A+1}{2} = A+1.$$
Hence we get $$\frac{\Psi}{\sqrt{\ell_w}}
= \underbrace{\frac{\Psi}{\sqrt{\ell_p}}}_{\xrightarrow{\hspace{0.1cm}d\hspace{0.1cm}} \mathcal{N}(0,\sigma_p^2)} 
\cdot ~ \underbrace{\sqrt{\frac{\ell_p}{\ell_w}}}_{\xrightarrow{\hspace{0.1cm}\Prob\hspace{0.1cm}}\frac{1}{\sqrt{A+1}}} \xrightarrow{\hspace{0.2cm} d \hspace{0.2cm}} \mathcal{N}\left(0,\sigma_w^2\right),$$ with $$\sigma_w^2 = \frac{\sigma_p^2}{A+1} = \frac{1}{12}\frac{A^2-1}{A+1}=\frac{A-1}{12}.$$

\vspace{1cm}
\begin{sloppypar}
\printbibliography

@book{Aigner2013,
	year = 2013,
	publisher = {Springer International Publishing},
	author = {M. Aigner},
	title = {Markov{\textquotesingle}s Theorem and 100 Years of the Uniqueness Conjecture}
}

@book{Einsiedler2011,
  doi = {10.1007/978-0-85729-021-2},
  url = {https://doi.org/10.1007/978-0-85729-021-2},
  year = {2011},
  publisher = {Springer London},
  author = {Manfred Einsiedler and Thomas Ward},
  title = {Ergodic Theory: with a view towards Number Theory}
}

@incollection {Kontorovich2016,
    AUTHOR = {Kontorovich, Alex},
     TITLE = {Applications of thin orbits},
 BOOKTITLE = {Dynamics and analytic number theory},
    SERIES = {London Math. Soc. Lecture Note Ser.},
    VOLUME = {437},
     PAGES = {289--317},
 PUBLISHER = {Cambridge Univ. Press, Cambridge},
      YEAR = {2016},
      ISBN = {978-1-107-55237-1},
   MRCLASS = {11J70 (11-02 37A45)},
  MRNUMBER = {3618792},
}

@article{Guivarch1993,
  doi = {10.24033/asens.1666},
  url = {https://doi.org/10.24033/asens.1666},
  year = {1993},
  publisher = {Societe Mathematique de France},
  volume = {26},
  number = {1},
  pages = {23--50},
  author = {Yves Guivarc{\textquotesingle}h and Yves Le Jan},
  title = {Asymptotic winding of the geodesic flow on modular surfaces and continuous fractions},
  journal = {Annales scientifiques de l{\textquotesingle}{\'{E}}cole normale sup{\'{e}}rieure}
}

@article{Series1985,
  doi = {10.1112/jlms/s2-31.1.69},
  url = {https://doi.org/10.1112/jlms/s2-31.1.69},
  year = {1985},
  publisher = {Wiley},
  volume = {2-31},
  number = {1},
  pages = {69--80},
  author = {Caroline Series},
  title = {The Modular Surface and Continued Fractions},
  journal = {Journal of the London Mathematical Society}
}

@article{BurrinVonEssen,
  title = {Windings of Prime Geodesics},
  volume = {2024},
  ISSN = {1687-0247},
  url = {http://dx.doi.org/10.1093/imrn/rnae225},
  DOI = {10.1093/imrn/rnae225},
  number = {22},
  journal = {International Mathematics Research Notices},
  publisher = {Oxford University Press (OUP)},
  author = {Burrin,  Claire and von Essen,  Flemming},
  year = {2024},
  pages = {13931–13963}
}

@book{Rademacher1972,
  doi = {10.5948/upo9781614440161},
  url = {https://doi.org/10.5948/upo9781614440161},
  year = {1972},
  publisher = {American Mathematical
		                    Society},
  author = {Hans Rademacher and Emil Grosswald},
  title = {Dedekind Sums}
}

@article{Sarnak2010,
author = {Peter Sarnak},
title = {{Linking Numbers of Modular Knots}},
volume = {8},
journal = {Communications in Mathematical Analysis},
number = {2},
publisher = {Mathematical Research Publishers},
pages = {136 -- 144},
year = {2010},
}

@book{Calegari2009stable,
  title={scl},
  author={Calegari, Danny},
  isbn={9784931469532},
  url={https://books.google.ch/books?id=TGvHPwAACAAJ},
  year={2009},
  publisher={Mathematical Society of Japan}
}

@article{Mozzochi2012,
  doi = {10.1007/s11856-012-0161-6},
  url = {https://doi.org/10.1007/s11856-012-0161-6},
  year = {2012},
  publisher = {Springer Science and Business Media {LLC}},
  volume = {195},
  number = {1},
  pages = {71--95},
  author = {C. J. Mozzochi},
  title = {Linking numbers of modular geodesics},
  journal = {Israel Journal of Mathematics}
}

@article{Kelmer2012,
  title = {Quadratic irrationals and linking numbers of modular knots},
  volume = {6},
  ISSN = {1930-532X},
  url = {http://dx.doi.org/10.3934/jmd.2012.6.539},
  DOI = {10.3934/jmd.2012.6.539},
  number = {4},
  journal = {Journal of Modern Dynamics},
  publisher = {American Institute of Mathematical Sciences (AIMS)},
  author = {Kelmer,  Dubi},
  year = {2012},
  pages = {539–561}
}

@article{Vardi1993,
  title={Dedekind sums have a limiting distribution},
  author={Ilan Vardi},
  journal={International Mathematics Research Notices},
  year={1993},
  pages={1-12},
  url={https://api.semanticscholar.org/CorpusID:125293536}
}

@book{Gut2013,
  title = {Probability: A Graduate Course},
  ISBN = {9781461447085},
  ISSN = {1431-875X},
  url = {http://dx.doi.org/10.1007/978-1-4614-4708-5},
  DOI = {10.1007/978-1-4614-4708-5},
  journal = {Springer Texts in Statistics},
  publisher = {Springer New York},
  author = {Gut,  Allan},
  year = {2013}
}

@article{Lalley1987,
  title={Distribution of periodic orbits of symbolic and Axiom A flows},
  author={Steven P. Lalley},
  journal={Advances in Applied Mathematics},
  year={1987},
  volume={8},
  pages={154-193},
  url={https://api.semanticscholar.org/CorpusID:14261130}
}

@unpublished{Casarotto2007,
  title={Markov Chains and the Ergodic Theorem},
  year={2007},
  author={Chad Casarotto}, 
}

@article{BourgainKontorovich2017,
  title = {Beyond expansion II: low-lying fundamental geodesics},
  volume = {19},
  ISSN = {1435-9863},
  url = {http://dx.doi.org/10.4171/JEMS/694},
  DOI = {10.4171/jems/694},
  number = {5},
  journal = {Journal of the European Mathematical Society},
  publisher = {European Mathematical Society - EMS - Publishing House GmbH},
  author = {Bourgain,  Jean and Kontorovich,  Alex},
  year = {2017},
  pages = {1331–1359}
}

@article{Artin1924,
  title = {Ein mechanisches System mit quasiergodischen Bahnen},
  volume = {3},
  number = {1},
  journal = {Abhandlungen aus dem Mathematischen Seminar der Universit\"{a}t Hamburg},
  publisher = {Springer Science and Business Media LLC},
  author = {Artin,  Emil},
  year = {1924},
  pages = {170–175}
}

@book{Berstel2008,
  title = {Combinatorics on Words},
  ISBN = {9781470417734},
  ISSN = {2472-5145},
  url = {http://dx.doi.org/10.1090/crmm/027},
  DOI = {10.1090/crmm/027},
  journal = {CRM Monograph Series},
  publisher = {American Mathematical
                    Society},
  author = {Berstel,  Jean and Lauve,  Aaron and Reutenauer,  Christophe and Saliola,  Franco},
  year = {2008},
  shorthand = {BLRS08}
}

@article{BasmajianSuzziValli2022,
  title = {Combinatorial growth in the modular group},
  volume = {16},
  ISSN = {1661-7215},
  url = {http://dx.doi.org/10.4171/GGD/667},
  DOI = {10.4171/ggd/667},
  number = {2},
  journal = {Groups,  Geometry,  and Dynamics},
  publisher = {European Mathematical Society - EMS - Publishing House GmbH},
  author = {Basmajian,  Ara and Suzzi Valli,  Robert},
  year = {2022},
  pages = {683–703}
}

@misc{BasmajianLiu2023,
  doi = {10.48550/ARXIV.2311.16041},
  url = {https://arxiv.org/abs/2311.16041},
  author = {Basmajian,  Ara and Liu,  Mingkun},
  title = {Low-lying geodesics on the modular surface and necklaces},
  publisher = {arXiv},
  year = {2023},
  copyright = {arXiv.org perpetual,  non-exclusive license}
}

@article{Burrin2023,
  title = {The Manin–Drinfeld theorem and the rationality of Rademacher symbols},
  volume = {34},
  ISSN = {2118-8572},
  url = {http://dx.doi.org/10.5802/jtnb.1225},
  DOI = {10.5802/jtnb.1225},
  number = {3},
  journal = {Journal de théorie des nombres de Bordeaux},
  publisher = {Cellule MathDoc/Centre Mersenne},
  author = {Burrin,  Claire},
  year = {2023},
  pages = {739–753}
}

@inbook{Ghys2007,
  title = {Knots and dynamics},
  ISBN = {9783985475360},
  url = {http://dx.doi.org/10.4171/022-1/11},
  DOI = {10.4171/022-1/11},
  booktitle = {Proceedings of the International Congress of Mathematicians Madrid,  August 22–30,  2006},
  publisher = {EMS Press},
  author = {Ghys,  Étienne},
  year = {2007},
  pages = {247–277}
}

@article{DIT2017,
  title = {Modular cocycles and linking numbers},
  volume = {166},
  ISSN = {0012-7094},
  url = {http://dx.doi.org/10.1215/00127094-3793032},
  DOI = {10.1215/00127094-3793032},
  number = {6},
  journal = {Duke Mathematical Journal},
  publisher = {Duke University Press},
  author = {Duke,  William and Imamoğlu,  \"{O}zlem and Tóth,  Árpád},
  year = {2017}
}

@book{Milnor1972,
  title = {Introduction to Algebraic K-Theory. (AM-72)},
  ISBN = {9781400881796},
  url = {http://dx.doi.org/10.1515/9781400881796},
  DOI = {10.1515/9781400881796},
  publisher = {Princeton University Press},
  author = {Milnor,  John},
  year = {1972}
}

@unpublished{kowalski-modcauchy,
  author       = {Emmanuel Kowalski},
  title        = {Examples of Mod-Cauchy Convergence},
  note         = {\url{https://people.math.ethz.ch/~kowalski/mod-cauchy.pdf}},
  year         = {2010}
}

@article{CalegariFujiwara2009,
  title = {Combable functions,  quasimorphisms,  and the central limit theorem},
  volume = {30},
  ISSN = {1469-4417},
  url = {http://dx.doi.org/10.1017/S0143385709000662},
  DOI = {10.1017/s0143385709000662},
  number = {5},
  journal = {Ergodic Theory and Dynamical Systems},
  publisher = {Cambridge University Press (CUP)},
  author = {Calegari, Danny and Fujiwara,  Koji},
  year = {2009},
  pages = {1343–1369}
}

@article{CantrellPollicott2022,
  title = {Comparison theorems for closed geodesics on negatively curved surfaces},
  volume = {16},
  ISSN = {1661-7215},
  url = {http://dx.doi.org/10.4171/GGD/671},
  DOI = {10.4171/ggd/671},
  number = {2},
  journal = {Groups,  Geometry,  and Dynamics},
  publisher = {European Mathematical Society - EMS - Publishing House GmbH},
  author = {Cantrell,  Stephen and Pollicott,  Mark},
  year = {2022},
  pages = {461–491}
}

@article{SpaldingVeselov2017,
  title = {Lyapunov spectrum of Markov and Euclid trees},
  volume = {30},
  ISSN = {1361-6544},
  url = {http://dx.doi.org/10.1088/1361-6544/aa88ff},
  DOI = {10.1088/1361-6544/aa88ff},
  number = {12},
  journal = {Nonlinearity},
  publisher = {IOP Publishing},
  author = {Spalding, Kathryn and Veselov,  Alexander Petrovich},
  year = {2017},
  pages = {4428–4453}
}

@misc{BIH2025,
  doi = {10.48550/ARXIV.2503.16343},
  url = {https://arxiv.org/abs/2503.16343},
  author = {Bengoechea,  Paloma and Herrero,  Sebastián and Imamoğlu,  \"{O}zlem},
  title = {A Lyapunov exponent attached to modular functions},
  publisher = {arXiv},
  year = {2025},
  copyright = {arXiv.org perpetual,  non-exclusive license}
}

@article{PollicottSharp,
 ISSN = {00029947},
 URL = {http://www.jstor.org/stable/117603},
 author = {Mark Pollicott and Richard Sharp},
 journal = {Transactions of the American Mathematical Society},
 number = {2},
 pages = {473--499},
 publisher = {American Mathematical Society},
 title = {Comparison Theorems and Orbit Counting in Hyperbolic Geometry},
 urldate = {2025-07-11},
 volume = {350},
 year = {1998}
}

@article{Einsiedler2009,
  title = {Distribution of periodic torus orbits on homogeneous spaces},
  volume = {148},
  ISSN = {0012-7094},
  url = {http://dx.doi.org/10.1215/00127094-2009-023},
  DOI = {10.1215/00127094-2009-023},
  number = {1},
  journal = {Duke Mathematical Journal},
  publisher = {Duke University Press},
  author = {Einsiedler,  Manfred and Lindenstrauss,  Elon and Michel,  Philippe and Venkatesh,  Akshay},
  year = {2009},
    shorthand = {ELMV09}
}

@article{Coelho1990,
  title = {Central limit asymptotics for shifts of finite type},
  volume = {69},
  ISSN = {1565-8511},
  url = {http://dx.doi.org/10.1007/BF02937307},
  DOI = {10.1007/bf02937307},
  number = {2},
  journal = {Israel Journal of Mathematics},
  publisher = {Springer Science and Business Media LLC},
  author = {Coelho,  Zaqueu and Parry,  William},
  year = {1990},
  month = jun,
  pages = {235–249}
}

@article {PollicottSharp1994,
    AUTHOR = {Pollicott, Mark and Sharp, Richard},
     TITLE = {Rates of recurrence for {${\bf Z}^q$} and {${\bf R}^q$}
              extensions of subshifts of finite type},
   JOURNAL = {J. London Math. Soc. (2)},
  FJOURNAL = {Journal of the London Mathematical Society. Second Series},
    VOLUME = {49},
      YEAR = {1994},
    NUMBER = {2},
     PAGES = {401--416},
      ISSN = {0024-6107,1469-7750},
   MRCLASS = {54H20 (58F03)},
  MRNUMBER = {1260121},
MRREVIEWER = {Peter\ Walters},
       DOI = {10.1112/jlms/49.2.401},
       URL = {https://doi.org/10.1112/jlms/49.2.401},
}
\end{sloppypar}
\end{document}